\documentclass[11pt]{amsart}
\usepackage{amsfonts}
\usepackage[latin9]{inputenc}
\usepackage{amsmath}
\usepackage{amssymb}
\usepackage{breqn}
\usepackage{url}
\usepackage{graphicx}
\usepackage[pdfstartview=FitH]{hyperref}
\usepackage{amsthm}
\usepackage{authblk}
\usepackage{hyperref}
\usepackage{url}
\usepackage{cleveref}
\usepackage{verbatim}

\crefformat{section}{\S#2#1#3} 
\crefformat{subsection}{\S#2#1#3}
\crefformat{subsubsection}{\S#2#1#3}

\newcommand{\im}{\operatorname{Im}}
\newcommand{\Ker}{\operatorname{Ker}}

\newcommand{\X}{\operatorname{X}}
\newcommand{\B}{\operatorname{\mathbb{E}}}
\newcommand{\Sym}{\operatorname{Sym}}
\newcommand{\sgn}{\operatorname{sgn}}
\newcommand{\id}{\operatorname{id}}
\newcommand{\tr}{\operatorname{tr}}

\newcommand{\dist}{\operatorname{dist}}
\newcommand{\Confdim}{\operatorname{Confdim}}

\makeatletter

\begin{document}
\newtheorem{theorem}{Theorem}[section]
\newtheorem{lemma}[theorem]{Lemma}
\newtheorem{claim}[theorem]{Claim}
\newtheorem{proposition}[theorem]{Proposition}
\newtheorem{corollary}[theorem]{Corollary}
\theoremstyle{definition}
\newtheorem{definition}[theorem]{Definition}
\newtheorem{observation}[theorem]{Observation}
\newtheorem{example}[theorem]{Example}
\newtheorem{remark}[theorem]{Remark}

\title{Garland's method with Banach coefficients}
\author{Izhar Oppenheim}
\thanks{The author partially supported by ISF grant no. 293/18.}
\newcommand{\Addresses}{{
  \bigskip
  \footnotesize
IZHAR OPPENHEIM, \textsc{Department of Mathematics, Ben-Gurion University of the Negev, Be'er Sheva 84105, Israel}\par\nopagebreak
  \textit{E-mail address:} \texttt{izharo@bgu.ac.il}
}}

\maketitle
\begin{abstract}
We prove a Banach version of Garland's method of proving vanishing of cohomology for groups acting on simplicial complexes. The novelty of this new version is that our new condition applies to every reflexive Banach space. 

This new version of Garland's method allows us to deduce several criteria for vanishing of group cohomology with coefficients in several classes of Banach spaces (uniformly curved spaces, Hilbertian spaces and $L^p$ spaces). 

Using these new criteria, we improve recent results regarding Banach fixed point theorems for random groups in the triangular model and give a sharp lower bound for the conformal dimension of the boundary of such groups. Also, we derive new criteria for group stability with respect to p-Schatten norms.
\end{abstract} 

\section{Introduction}

Let $X$ be a locally finite, pure $n$-dimensional simplicial complex and $\Gamma$ be a locally compact, unimodular group acting cocompactly and properly on $X$. Under the assumption that $X$ is an affine building, Garland \cite{Gar} gave a local criterion for the vanishing of the equivariant k-th cohomology for any unitary representation of $\Gamma$ and any $1 \leq k \leq n-1$. His approach was later generalized by Ballmann and {\'S}wiatkowski \cite{BS} to all simplicial complexes and this generalization is sometimes referred to as ``Garland's method''. There have been several generalizations of this method that considered the case where $\pi$ is an isometric representation on a Banach space - see \cite{Nowak}, \cite{Koivisto}, \cite{OppFixed}. However, all these generalizations gave somewhat weak results when applied to examples. For example, when considering vanishing of cohomology over $L^p$ spaces, the results of \cite{Nowak} could not show vanishing of cohomology every $1 < p < \infty$ neither for $\widetilde{A}_2$ groups nor for random groups (see \cite[Theorems 5.1, 6.2]{Nowak}).

We note that Garland's original work referred to affine buildings and in this set-up strong results regarding vanishing of cohomologies with Banach coefficients are known: See \cite{Laff}, \cite{Liao} and \cite{LdlSW} for results regarding vanishing of the first cohomology and see \cite{Opp}, \cite{LO} for results regarding vanishing of higher cohomologies. However, much less is known when one considers the less structured setting of a group acting on a simplicial complex without assuming the extra structure of an affine building. 

Recently, the results for vanishing of the first cohomology of random groups with coefficeints in Banach spaces were improved: First, Dru\c{t}u and Mackay \cite{DrutuM} proved vanishing of the first cohomology for random groups for $L^p$ spaces. Second, considering random groups in the triangular model, de Laat and de la Salle \cite{LaatSalle} gave a criterion for vanishing of the first cohomology for a group acting on a two dimensional simplicial complex that was applicable to all uniformly curved Banach spaces (and in particular, to all $L^p$ spaces). 

The observation of de Laat and de la Salle was that one can get much stronger results than in previous works if the assumption of the spectral gap in the links is replaced with the assumption of a two-sided spectral gap (or equivalently with the assumption of contraction of the random walk operator).  Using this insight and the ideas of Nowak \cite{Nowak}, we rework Garland's method under the assumption of two-sided spectral gaps in the links and derive a very general vanishing criteria that are applicable to all uniformly curved Banach spaces (and, in part, to all reflexive Banach spaces). We give two applications for our result:
\paragraph*{\textbf{Fixed point properties of random groups}} Applying our vanishing result to random groups in the triangular model improve on the results of de Laat and de la Salle when considering fixed point properties with respect to $L^p$ spaces. As a result, we derive a \emph{sharp} lower bound for the conformal dimension of the boundary of these groups that was not achieved in previous works. Namely, in previous works (\cite{DrutuM}, \cite{LaatSalle}) it was shown that with high probability, this conformal dimension is contained in an interval between $C \sqrt{\log m}$ and  $C ' \log (m)$ (where $m$ is a parameter of the model - see exact formulation below). Our work shows that in fact the conformal dimension is in an interval of the form $C '' \log m$ and  $C ' \log (m)$ and thus our result is sharp. We note that as far as we understand, the proof methods in \cite{DrutuM}, \cite{LaatSalle} can not be improved to yields such a sharp bound. 
\paragraph*{\textbf{Group stability with respect to $p$-Schatten norms}} By a result of \cite{DGLT}, vanishing of the second cohomology for Hilbertian spaces implies stability with respect to $p$-Schatten norms (see definitions below). Thus, our new criteria for vanishing of the second cohomology gives new criteria for group stability.

\subsection{New criteria for vanishing of cohomology with Banach coefficients}
In order to state our results, we will need the following notation: For every simplex $\tau \in X (k)$ denote $X_\tau$ to be the link of $\tau$ and $M_\tau, A_\tau : \ell^2 (X_\tau (0)) \rightarrow  \ell^2 (X_\tau (0))$ to be the following operators: $M_\tau$ is the orthogonal projection on the subspace of constant functions in $\ell^2 (X_\tau (0))$ and $A_\tau$ in the random walk operator on the $1$-skeleton of $X_\tau$. With these notations, we prove the following:

\begin{theorem}
\label{local criterion thm - intro}
Let $\B$ be a reflexive Banach space, $X$ a locally finite, pure $n$-dimensional simplicial complex and $\Gamma$ a locally compact, unimodular group acting cocompactly and properly on $X$. 

For $1 \leq k \leq n-1$, if 
$$\max_{\tau \in X (k-1)} \Vert (A_\tau (I-M_\tau) \otimes \id_{\B}) \Vert_{B (\ell^2 (X_\tau (0) ; \B))} < \frac{1}{k+1},$$
then for every continuous isometric representation $\pi$ of $\Gamma$ on $\B$, $H^k (X, \pi) = 0$. 
\end{theorem}

\begin{remark}
A result of the same flavor was given in \cite[Theorem B]{LaatSalle} by de Laat and de la Salle for the vanishing of the first cohomology for groups acting on two-dimensional simplicial complexes. We note that our Theorem improves on \cite[Theorem B]{LaatSalle} even when considering only vanishing of the first cohomology: First, our Theorem holds for any reflexive Banach space, while \cite[Theorem B]{LaatSalle} is only applicable for super-reflexive spaces. Second, in terms of parameters, the contraction condition of $A_\tau (I-M_\tau) \otimes \id_{\B}$ does not depend on the Banach space, but only on $k$ (as opposed to \cite[Theorem B]{LaatSalle}). Last, our proof is simpler in the regard that it does not use the $p$-Laplacian or any uniform convexity arguments.
\end{remark}

Theorem \ref{local criterion thm - intro} is easily applicable in the setting uniformly curved Banach spaces (see Definition \ref{uc def}) such as (commutative and non-commutative) $L^p$ spaces and more generally strictly $\theta$-Hilbertian spaces (see exact definition in \cref{subsect Strictly theta-Hilbertian spaces}). Namely, for a uniformly curved Banach space we can deduce vanishing of cohomology based the fact that links are spectral expanders. Before stating these type of results, we recall the relevant terminology: Let $(V,E)$ be a connected finite graph and let $A$ be the random walk operator on this graph. Recall that $A$ is a self-adjoint operator and has the eigenvalue $1$ with multiplicity $1$. For a constant $\lambda$, the graph $(V,E)$ is called a \textit{one-sided $\lambda$-spectral expander} if the second largest eigenvalue of $A$ is $\leq \lambda$. The graph $(V,E)$ is called a \textit{two-sided $\lambda$-spectral expander} if the spectrum of $A$ is contained in the interval $[-\lambda, \lambda] \cup \lbrace 1 \rbrace$.

\begin{theorem}[Informal, see Proposition \ref{vanishing of k-coho prop for uni curv} and Theorem \ref{criterion coho vanish for uc} for explicit formulations]
\label{uniformly curved thm intro}
Let $\B$ be a uniformly curved Banach space. There are positive constants $\lbrace \lambda_k (\B) >0 : k \in \mathbb{N} \rbrace$ such that the for every locally finite, pure $n$-dimensional simplicial complex $X$, every locally compact, unimodular group $\Gamma$ acting cocompactly and properly on $X$ the following holds:
\begin{enumerate}
\item For every $1 \leq k \leq n-1$, if there is $0 < \lambda < \lambda_k (\B)$ such that for every $\tau \in X(k-1)$, the one skeleton of $X_\tau$ is a two-sided $\lambda $-spectral expander, then for every continuous isometric representation $\pi$ of $\Gamma$ on $\B$, $H^k (X, \pi) = 0$.
\item For every $1 \leq k \leq n-1$,  if there is $0 <  \lambda  < \lambda_k (\B)$ such that for every $\tau \in X(n-2)$, the one skeleton of $X_\tau$ is a two-sided $\frac{\lambda }{1+ (n-k-1) \lambda }$-spectral expander, then for every continuous isometric representation $\pi$ of $\Gamma$ on $\B$, $H^k (X, \pi) = 0$.
\item For every $1 \leq k < n - \frac{1}{\lambda_k (\B)}$, if there is $0 <  \lambda  < \lambda_k (\B)$ such that for every $\tau \in X(n-2)$, the one skeleton of $X_\tau$ is a one-sided $\frac{\lambda }{1+ (n-k-1) \lambda }$-spectral expander, then for every continuous isometric representation $\pi$ of $\Gamma$ on $\B$, $H^k (X, \pi) = 0$. 
\end{enumerate}
\end{theorem}

Specifying the above Corollary to $\theta$-Hilbertian spaces reads as follows:
\begin{corollary}
\label{theta Hil coro - intro}
Let $X$ be a locally finite, pure $n$-dimensional simplicial complex such that all the links of $X$ of dimension $\geq 1$ are connected and $\Gamma$ be a locally compact, unimodular group acting cocompactly and properly on $X$. Also let $0 < \theta_0 \leq 1$, $1 \leq k \leq n-1$, $0 < \lambda <  (\frac{1}{2(k+1)})^{\frac{1}{\theta_0}}$ be constants. Denote $\mathcal{E}_{\theta_0}$ to be the smallest class of Banach spaces that contains all strictly $\theta$-Hilbertian Banach spaces for all $\theta_0 \leq \theta \leq 1$ and is closed under subspaces, quotients, $\ell^2$-sums and ultraproducts of Banach spaces.
\begin{enumerate}
\item If for every $\tau \in X(k-1)$, the one skeleton of $X_\tau$ is a two-sided $\lambda $-spectral expander, then for every $\B \in \mathcal{E}_{\theta_0}$ and every continuous isometric representation $\pi$ of $\Gamma$ on $\B$, $H^k (X, \pi) = 0$.
\item If for every $\tau \in \Sigma (n-2, \Gamma)$, the one-skeleton of $X_\tau$ is a two-sided $\frac{\lambda}{1 + (n-k-1) \lambda}$-spectral expander, then for every $\B \in \mathcal{E}_{\theta_0}$ and every continuous isometric representation $\pi$ of $\Gamma$ on $\B$, $H^k(X,\pi) = 0$.
\item If $k \leq n - \frac{1}{\lambda}$ and for every $\tau \in \Sigma (n-2, \Gamma)$, the one-skeleton of $X_\tau$ is a \emph{one-sided} $\frac{\lambda}{1 + (n-k-1) \lambda}$-spectral expander, then for every $\B \in \mathcal{E}_{\theta_0}$ and every continuous isometric representation $\pi$ of $\Gamma$ on $\B$, $H^k(X,\pi) = 0$.
\end{enumerate}
\end{corollary}

\begin{remark}
\label{equi coho equals group coho rmrk}
In all the Theorems above, we gave criteria for vanishing of the equivariant cohomology. We recall that given that the simplicial complex $X$ is aspherical, it holds that $H^k (X, \pi) = H^k (\Gamma, \pi)$ (the proof of this can be found for instance in \cite[Chapter 7, Section 7]{Brown}) and thus under this additional assumption it follows that the criteria given above imply that $H^k (\Gamma, \pi) =0$.  
\end{remark}

\subsection{Application to random groups}

An immediate application of our criteria above is improving de Laat and de la Salle's results regarding random groups in the triangular model. The triangular model for random groups, denoted $\mathcal{M} (m,d)$ is defined as follows: For a fixed density $d \in (0,1)$, a group in $\mathcal{M} (m,d)$ is a finitely presented group of the form $\Gamma = \langle S \vert R \rangle$, where $\vert S \vert = m$ ($S \cap S^{-1} = \emptyset$) and $R$ is a set of cyclically reduced relators of length $3$ chosen uniformly among all subsets of cardinality $\lfloor (2m-1)^{3d} \rfloor$. A property P for groups is said to hold with overwhelming probability in this model if
$$\lim_{m \rightarrow \infty} \mathbb{P} (\Gamma \text{ in } \mathcal{M} (m,d) \text{ has P}) =1.$$

Below, we will also use the the binomial triangular model that is closely related model to the triangular model. The  binomial triangular model, denoted $\Gamma (m, \rho)$, is defined as follows:  A group in $\Gamma (m, \rho)$ is a finitely presented group of the form $\Gamma = \langle S \vert R \rangle$, where $\vert S \vert = m$ and $R$ is a set of cyclically reduced relators of length $3$, where each relator is chosen independently with probability $\rho$. We mention this model, since it would

The triangular model for random groups was introduced by Zuk \cite{Zuk} who showed that when $d > \frac{1}{3}$, then property (T) holds for groups in $\mathcal{M} (m,d)$ with overwhelming probability. The work of de Laat and de la Salle \cite{LaatSalle} (that followed the work of Dru\c{t}u and Mackay \cite{DrutuM}) generalized the result of Zuk to the setting of uniformly curved Banach spaces. In order to explain this generalization, we recall that by a classical result of Delorme-Guichardet, finitely generated discrete group $\Gamma$ has property (T) if and only if it has property $(FH)$, i.e., if and only if every affine isometric action of $\Gamma$ on a Hilbert space admits a fixed point. Property (FH) is readily generalized to the Banach setting as follows: For a Banach space $\B$, a group $\Gamma$ is said to have property $(F_{\B})$ if every continuous affine isometric action of $\Gamma$ on $\B$ admits a fixed point. Also, a group $\Gamma$ is said to have property $(F_{L^p})$ if it has property $(F_{\B})$ for every $L^p$ space $\B$. In their work, de Laat and de la Salle \cite{LaatSalle} showed that if $d > \frac{1}{3}$, then for every uniformly curved Banach space $\B$, property $(F_{\B})$ holds for groups in $\mathcal{M} (m,d)$ with overwhelming probability (their result is actually stronger - see Theorem \ref{LaatSalle thm C} stated below).

Our results above are stated in the language of vanishing of the equivariant cohomology from groups acting on simplicial complexes. The connection between fixed point properties and vanishing of cohomology readily follows from the following classical interpretation of group cohomology (See for instance the discussion in \cite[Section 2]{FVM}):
\begin{proposition}
Let $\Gamma$ be a topological group and $\B$ be a Banach space. the group $\Gamma$ has property $(F_{\B})$ if and only if for every continuous isometric representation $\pi$ of $\Gamma$ on $\B$, $H^1 (\Gamma, \pi) =0$.
\end{proposition}

The connection to our results regarding vanishing of cohomology described above is the following equivalence between fixed points and vanishing of the first cohomology: We recall that for a topological group $\Gamma$ and a Banach space $\B$ the following are equivalent:
\begin{itemize}
\item For every continuous isometric representation $\pi$ of $\Gamma$ on $\B$, $H^1 (\Gamma, \pi) =0$.
\item The group $\Gamma$ has property $(F_{\B})$. 
\end{itemize}

Above, we discussed the vanishing of equivariant cohomology and not group cohomology, but as noted in Remark \ref{equi coho equals group coho rmrk}, this is equivalent in the case of groups acting on aspherical complexes. For a random group $\Gamma$ in the model $\Gamma (m, \rho)$ (or in the model $\mathcal{M} (m,d)$), the Cayley complex of the group is $2$-dimensional simplicial complex that we will denote by $X_\Gamma$. We recall that the Cayley complex of a group is always simply connected and the action of a group on its Cayley complex is simply transitive on the vertices. In particular, since the group is finitely presented, the action is proper and cocompact. Thus, the vanishing of the first cohomology of $\Gamma$ is equivalent to the vanishing of the first equivariant cohomology for the action of $\Gamma$ on $X_\Gamma$. It follows that if we know that the links of $X_\Gamma$ are two-sided spectral expanders, we can deduce property $(F_{\B})$ for a uniformly curved Banach space $\B$ and a random group $\Gamma$ in the model $\Gamma (m, \rho)$, by applying Theorem \ref{uniformly curved thm intro} stated above. 

In \cite{LaatSalle}, it was proven that the links of $X_\Gamma$ for $\Gamma (m, \rho)$ are indeed two-sided spectral expanders:
\begin{proposition}\cite[Proposition 7.5]{LaatSalle}
\label{random grp local spec gap prop}
Let $\eta >0$ be a constant. There is a constant $C>0$ and a sequence $\lbrace u_m \rbrace_{m \in \mathbb{N}}$ tending to $0$ such that the following holds: Let $m \in \mathbb{N}$ and $\rho \in (0,m^{-1.42})$. Also let $\Gamma$ be a random group in the model $\Gamma (m, \rho)$ and $X_\Gamma$ its Cayley complex. If $\rho \geq \frac{(1+\eta) \log m}{8 m^2}$, then with probability $\geq 1- u_m$, the link of every vertex of $X_\Gamma$ is a $\sqrt{\frac{C}{\rho m^2}}$-two-sided spectral expander.
\end{proposition}

Combining this Proposition with Theorem \ref{uniformly curved thm intro} above, we can reprove the following Theorem that already appeared in \cite{LaatSalle} (see \cite[Theorem 7.3]{LaatSalle}):
\begin{theorem}
Let $\eta  '>0$ and $\rho \in (0,m^{-1.42})$ be constants and let $C$ be the constant that appears in Proposition \ref{random grp local spec gap prop}. Assume that $\rho \geq \frac{(1+\eta ') \log m}{8 m^2}$ and let $\Gamma$ be a random group in the model $\Gamma (m, \rho)$. Then there is a sequence $\lbrace u_m \rbrace_{m \in \mathbb{N}}$ tending to $0$ such that for uniformly curved Banach space $\B$ with $\lambda_1 (\B) \geq \sqrt{\frac{C}{\rho m^2}}$ (where $\lambda_1 (\B)$ as in Theorem \ref{uniformly curved thm intro}) it holds that $\Gamma$ has property $(F_{\B})$ with probability $\geq 1- u_m$.
\end{theorem}

As in \cite{LaatSalle}, using the fact that the fixed point property passes to quotients, we can also recast this Theorem in the triangular model (see further details in \cite[Section 7]{LaatSalle}) and reprove \cite[Theorem C]{LaatSalle}:
\begin{theorem}
\label{LaatSalle thm C}
Let $0 < \eta <2$, $d > \frac{1}{3} + \frac{\log \log m - log (2- \eta)}{3 \log m}$ be constants.  Also let $\Gamma$ be a random group in the model $\mathcal{M} (m, d)$. Then there is a sequence $\lbrace u_m \rbrace_{m \in \mathbb{N}}$ tending to $0$ and a constant $C>0$ such that for uniformly curved Banach space $\B$ with $\lambda_1 (\B) \geq \sqrt{\frac{C}{(2m-1)^{3d-1}}}$ (where $\lambda_1 (\B)$ as in Theorem \ref{uniformly curved thm intro}) it holds that $\Gamma$ has property $(F_{\B})$ with probability $\geq 1- u_m$.
\end{theorem}

Combining this Theorem with Corollary \ref{criterion coho vanish for L^p} leads to a stronger result than the one stated in \cite{LaatSalle} (and in \cite{DrutuM}) when considering $L^p$ spaces. Namely, applying Corollary \ref{criterion coho vanish for L^p} yields the following:
\begin{theorem}
Let $0 < \eta <2$, $d > \frac{1}{3} + \frac{\log \log m - log (2- \eta)}{3 \log m}$ be constants.  Also let $\Gamma$ be a random group in the model $\mathcal{M} (m, d)$. Then there is a sequence $\lbrace u_m \rbrace_{m \in \mathbb{N}}$ tending to $0$ and a constant $C>0$ such for 
$$2 \leq p \leq \frac{1}{2} (3d-1) \log (2m-1) - \frac{1}{2} \log C$$ 
it holds that $\Gamma$ has property $(F_{L^p})$ with probability $\geq 1- u_m$.
\end{theorem}

As a Corollary, we improve the bound on the conformal dimension of random groups in the triangular model stated in \cite[Corollary E]{LaatSalle}. Namely, by a Theorem by Bourdon (see \cite{Bourdon}), if for a given $2 \leq p$, a hyperbolic group $\Gamma$ has property $(F_{L^p})$, then the conformal dimension of $\partial_\infty \Gamma$ is $\geq p$. Thus, we get
\begin{theorem}
Let $0 < \eta <2$, $d > \frac{1}{3} + \frac{\log \log m - log (2- \eta)}{3 \log m}$ be constants.  Also let $\Gamma$ be a random group in the model $\mathcal{M} (m, d)$. Then there is a sequence $\lbrace u_m \rbrace_{m \in \mathbb{N}}$ tending to $0$ and a constant $C>0$ such  
$$\frac{1}{2} (3d-1) \log (2m-1) - \frac{1}{2} \log C \leq \Confdim (\partial_\infty \Gamma)$$
with probability $\geq 1- u_m$. 

In particular, for $d \in (\frac{1}{3}, \frac{1}{2})$ and a group $\Gamma$ in $\mathcal{M} (m, d)$ it holds with overwhelming probability that 
$$\frac{1}{2} (3d-1) \log (2m-1) - \frac{1}{2} \log C \leq \Confdim (\partial_\infty \Gamma).$$
\end{theorem}

\begin{remark}
The above Theorem gives a sharp bound on the conformal dimension of the boundary: Indeed, in \cite[Proposition 10.6]{DrutuM} it was shown that for $d \in (\frac{1}{3}, \frac{1}{2})$ and a group $\Gamma$ in $\mathcal{M} (m, d)$ it holds with overwhelming probability that 
$$\Confdim (\partial_\infty \Gamma) \leq \frac{30}{2d-1} \log (2m-1).$$
\end{remark}

\subsection{Application to group stability}

Group stability have received much attention in recent years (see for instance \cite{GR}, \cite{AP}, \cite{BLT},\\ \cite{DGLT}, \cite{BL}) partly due to its connection to questions of group approximation (see for instance \cite{DGLT}). In \cite{DGLT}, it was shown that, under some assumptions, group stability can be deduced for a group via the vanishing of its second cohomology. Another application of our work is providing a criterion for $p$-norm stability (stability with respect to the $p$-Schatten norm). In order to state this application, we first give the needed definitions and results from \cite{DGLT}.

Let $\Gamma$ be a finitely presented group $\Gamma = \langle S \vert R \rangle$, with $R \subseteq \mathbb{F}_S$ - the free group on $S$ and $\vert R \vert < \infty$. Any map $\phi : S \rightarrow U(n)$ uniquely determines a homomorphism $\phi :\mathbb{F}_S \rightarrow U(n)$ which we will also denote by $\phi$. 

Given a distance $\dist_n$ on $U(n)$, the group $\Gamma$ is called \textit{$\mathcal{G}=(U(n), \dist_n)$-stable} if for every $\varepsilon >0$ there exists $\delta >0$ such that for every $n \in \mathbb{N}$, if $\phi : S \rightarrow U(n)$ is a map with 
$$\sum_{r \in R} \dist_n (\phi (r), \id_{U(n)} ) < \delta,$$ 
then there exists a homomorphism $\tilde{\phi}: \Gamma \rightarrow U(n)$ (or equivalently, a map $\tilde{\phi}  : S \rightarrow U(n)$ with $\sum_{r \in R} \dist_n (\tilde{\phi}  (r), \id_{U(n)} ) = 0$) with 
$$\sum_{s \in S} \dist_n (\phi (s), \tilde{\phi}  (s)) < \varepsilon.$$

For $1 \leq p < \infty$, the Schatten $p$-norm on $M_n (\mathbb{C})$ is defined by $\Vert T \Vert_p =  \left( \tr \vert T \vert^p \right)^{\frac{1}{p}}$, where $\vert T \vert = \sqrt{T^* T}$. When $p=2$, this is usually called the Frobenius norm. Denote $\dist_{n,p}$ to be the metric on $U(n)$ induced by this norm. Below, we will call a group $\Gamma$ \textit{$p$-norm stable} if it is stable with respect to  $\mathcal{G} = (U(n), \dist_{n,p})$.

We note that $(M_n (\mathbb{C}), \Vert . \Vert_p)$ is a non-commutative $L^p$ space and in particular, it is strictly $\theta$-Hilbertian with $\theta = 2 - \frac{2}{p}$ if $p \leq 2$ and $\theta =  \frac{2}{p}$ if $p \geq 2$. The discussion in \cite{DGLT} implies the following criterion for $p$-norm stability (see also \cite{GM} and \cite{LO}):
\begin{theorem}\cite[Theorem 5.1, Remark 5.2]{DGLT} 
Let $\Gamma$ be a finitely presented group and $0 < \theta_0 \leq 1$ be a constant. Denote $\mathcal{E}_{\theta_0}$ to be the smallest class of Banach spaces that contains all strictly $\theta$-Hilbertian Banach spaces for all $\theta_0 \leq \theta \leq 1$ and is closed under subspaces, $\ell^2$-sums and ultraproducts of Banach spaces. If for every for every continuous isometric representation $\pi$ of $\Gamma$ on $\B \in \mathcal{E}_{\theta_0}$ it holds that $H^2 (\Gamma, \pi) =0$, then $\Gamma$ is $p$-norm stable for every $1+ \frac{\theta_0}{2- \theta_0} \leq p \leq \frac{2}{\theta_0}$. 
\end{theorem}

Combining this Theorem with Corollary \ref{theta Hil coro - intro} and Remark \ref{equi coho equals group coho rmrk} immediately yields the following criterion for $p$-norm stability:
\begin{theorem}
\label{stability thm intro}
Let $X$ be a locally finite, pure $n$-dimensional aspherical simplicial complex with $n \geq 3$ such that all the links of $X$ of dimension $\geq 1$ are connected and $\Gamma$ be a finitely presented discrete group acting cocompactly and properly on $X$. Also let $0 < \theta_0 \leq 1$, $0 < \lambda <  (\frac{1}{6})^{\frac{1}{\theta_0}}$ be constants. Assume that \emph{one} of the following holds:
\begin{enumerate}
\item For every $\tau \in X(1)$, the one skeleton of $X_\tau$ is a two-sided $\lambda$-spectral expander.
\item For every $\tau \in X(n-2)$, the one-skeleton of $X_\tau$ is a two-sided $\frac{\lambda}{1 + (n-3) \lambda}$-spectral expander.
\item It holds that $2 \leq n - \frac{1}{\lambda}$ and for every $\tau \in X(n-2)$, the one-skeleton of $X_\tau$ is a \emph{one-sided} $\frac{\lambda}{1 + (n-3) \lambda}$-spectral expander.
\end{enumerate}
Then $\Gamma$ is $p$-norm stable for every $1+ \frac{\theta_0}{2- \theta_0} \leq p \leq \frac{2}{\theta_0}$.
\end{theorem}

Currently, we do not have new examples in which this Theorem improves previous results. One can take $X$ to be an affine building of a large dimension $n$, $\Gamma$ a lattice of the full BN-pair group of $X$ and apply Theorem \ref{stability thm intro} (3) to deduce $p$-norm stability (where $p$ depends on the thickness of the building and on $n$). However, as noted above, in the case where $X$ is a classical affine building, stronger results are given in \cite{LO}. \\ 

\paragraph*{\textbf{Organization}} 
This paper is organized as follows: In \cref{prelim sec}, we cover some needed preliminaries. In \cref{equiv coho sec}, we give the basic definitions regarding equivariant cohomology and prove a variation of Nowak's criterion for vanishing of cohomology. In \cref{local criteria sec}, we prove our local criteria for vanishing of Banach cohomology. \\

\paragraph*{\textbf{Acknowledgement}} I want to thank Mikael de la Salle for pointing our  ut that the in the setting of random groups, the results of this paper yield a sharp lower bound for the conformal dimension and for noting several errors in a preliminary draft of this paper. 

\section{Preliminaries}
\label{prelim sec}

\subsection{Vector valued $\ell^2$ spaces}
\label{Vector valued subsec}

Given a finite set $V$, a function  $m: V \rightarrow \mathbb{R}_+$ and a Banach space $\B$, we define the \textit{vector valued space} $\ell^2 (V, m ; \B)$ to be the space of functions $\phi : V \rightarrow \B$, with the norm 
$$\Vert \phi \Vert_{\ell^2 (V,m ; \B)} = \left( \sum_{v \in V} m(v) \vert \phi (v) \vert^2 \right)^{\frac{1}{2}},$$
where $\vert . \vert$ is the norm of $\B$. We denote $\ell^2 (V,m) = \ell^2 (V,m ; \mathbb{C})$ and recall that $\ell^2 (V,m)$ is also a Hilbert space with the inner-product
$$\langle \phi, \psi \rangle = \sum_{v  \in V} m(v) \phi (v) \overline{\psi (v)}.$$

Let $T: \ell^2 (V,m) \rightarrow \ell^2 (V,m)$ be a linear operator and $T_{v,u} \in \mathbb{C}$ be the constants such that for every $\phi \in \ell^2 (V,m)$ it holds that 
$$(T \phi) (v) = \sum_{u \in V} T_{v,u} \phi (u).$$
Define $T \otimes \id_{\B} : \ell^2 (V,m ; \B) \rightarrow \ell^2 (V,m ; \B)$ by the formula:
$$((T \otimes \id_{\B}) \phi) (v) = \sum_{u \in V} T_{v,u} \phi (u),$$
where $T_{v,u} \in \mathbb{C}$ are the same constants as above and $\phi \in \ell^2 (V,m ; \B)$. We denote $\Vert T \otimes \id_{\B} \Vert_{B( \ell^2 (V,m ; \B))}$ to be the operator norm of $T \otimes \id_{\B}$.

Following Pisier \cite{Pisier2}, we call an operator $T: \ell^2 (V,m) \rightarrow \ell^2 (V,m)$ \textit{fully contractive} if for every Banach space $\B$ it holds that $\Vert T \otimes \id_{\B} \Vert_{B( \ell^2 (V,m ; \B))} \leq 1$. 

\subsection{Uniformly curved Banach spaces}

Uniformly curved Banach spaces where introduced by Pisier in \cite{Pisier2}:
\begin{definition}
\label{uc def}
Let $\B$ be a Banach space. The space $\B$ is called uniformly curved if for every $0 < \varepsilon \leq 1$ there is $\delta >0$ such that for every space $\ell^2 (V,m)$ and every fully contractive linear operator $T:\ell^2 (V,m) \rightarrow \ell^2 (V,m)$, if $\Vert T \Vert_{B(\ell^2 (V,m))} \leq \delta$, then $\Vert T \otimes \id_{\B} \Vert_{B(\ell^2 (V,m ; \B))} \leq \varepsilon$. 
\end{definition}

The following Theorem is due to Pisier \cite{Pisier2}:
\begin{theorem}
\label{uni. cur. are s-reflex. thm}
Every uniformly curved Banach space is super-reflexive and in particular reflexive.
\end{theorem}

Given a monotone increasing function $\alpha: (0,1] \rightarrow  (0,1]$ such that 
$$\lim_{t \rightarrow 0^{+}} \alpha (t) = 0,$$
we denote 
$\mathcal{E}^{\text{u-curved}}_\alpha$ to be the class of all (uniformly curved) Banach spaces $\B$ such that for every space $\ell^2 (V,m)$ and every fully contractive linear operator $T:\ell^2 (V,m) \rightarrow \ell^2 (V,m)$, if $\Vert T \Vert_{B(\ell^2 (V,m))} \leq \delta$, then $\Vert T \otimes \id_{\B} \Vert_{B(\ell^2 (V,m ; \B))} \leq \alpha (\delta)$. 

\begin{proposition}
\label{L contractive prop}
Let $T :\ell^2 (V,m) \rightarrow \ell^2 (V,m)$ be a linear operator and $L\geq 1$, $0< \delta \leq 1$ be constants such that:
\begin{enumerate}
\item It holds that $\Vert T \Vert_{B(\ell^2 (V,m))} \leq \delta$.
\item For every Banach space $\B$, $\Vert T \otimes \id_{\B} \Vert_{B(\ell^2 (V,m ; \B))} \leq L$.
\end{enumerate}
Then for every monotone increasing function $\alpha: (0,1] \rightarrow  (0,1]$ such that $\lim_{t \rightarrow 0^{+}} \alpha (t) = 0$ and every $\B \in \mathcal{E}^{\text{u-curved}}_\alpha$, $\Vert T \otimes \id_{\B} \Vert_{B(\ell^2 (V,m ; \B))} \leq L \alpha (\delta)$.
\end{proposition}

\begin{proof}
We note that $\frac{1}{L} T$ is a fully contractive operator such that 
$$\Vert \frac{1}{L} T \Vert_{B(\ell^2 (V,m))} \leq \frac{\delta}{L}.$$ 
Thus, by the definition of $\mathcal{E}^{\text{u-curved}}_\alpha$ it follows for every $\B \in \mathcal{E}^{\text{u-curved}}_\alpha$ that 
$$\Vert (\frac{1}{L}) T \otimes \id_{\B} \Vert_{B(\ell^2 (V,m ; \B))} \leq \alpha (\frac{\delta}{L}),$$
and thus 
$$\Vert T \otimes \id_{\B} \Vert_{B(\ell^2 (V,m ; \B))} \leq L \alpha (\frac{\delta}{L}) \leq L \alpha (\delta),$$
where the last inequality is due to the fact that $L \geq 1$ and $\alpha$ is monotone increasing.
\end{proof}

We will also be interested in how $T \otimes \id_\B$ behaves under some operations - this is summed up in the following lemmas:
\begin{lemma}
\label{L2 norm stability}
Let $V$ be a finite set, $T$ a bounded operator on $\ell^2 (V, m)$ and $C >0$ constant. Let $\mathcal{E} = \mathcal{E} (C)$ be the class of Banach spaces defined as:
$$\mathcal{E} = \lbrace \B : \Vert T \otimes \id_\B \Vert_{B(\ell^2 (V, m ; \B))} \leq C \rbrace.$$
Then this class is closed under quotients, subspaces, $\ell^2$-sums and ultraproducts of Banach spaces, i.e., preforming any of these operations on Banach spaces in $\mathcal{E}$ yield a Banach space in $\mathcal{E}$. 
\end{lemma}

\begin{proof}
The fact that $\mathcal{E}$ is closed under quotients, subspaces and ultraproducts of Banach spaces was shown in \cite[Lemma 3.1]{Salle}. The fact that $\mathcal{E}$ is closed under $\ell^2$-sums is straight-forward and left for the reader.
\end{proof}

Applying Lemma \ref{L2 norm stability} on $\mathcal{E}^{\text{u-curved}}_\alpha$ defined above yields the following corollary:
\begin{corollary}
For any monotone increasing function $\alpha: (0,1] \rightarrow  (0,1]$ such that $\lim_{t \rightarrow 0^{+}} \alpha (t) = 0$, the class $\mathcal{E}^{\text{u-curved}}_\alpha$ defined above is closed under quotients, subspaces, $\ell^2$-sums and ultraproducts of Banach spaces.
\end{corollary}

\subsection{Strictly $\theta$-Hilbertian spaces}
\label{subsect Strictly theta-Hilbertian spaces}

Here we will describe a special class of uniformly curved Banach spaces that contains all (commutative and non-commutative) $L^p$ spaces. 

Two Banach spaces $\B_0, \B_1$ form a \textit{compatible pair} $(\B_0,\B_1)$ if they are continuously linear embedded in the same topological vector space. The idea of complex interpolation is that given a compatible pair $(\B_0,\B_1)$ and a constant $0 \leq \theta \leq 1$, there is a method to produce a new Banach space $[\B_0, \B_1]_\theta$ as a ``convex combination'' of $\B_0$ and $\B_1$. We will not review this method here, and the interested reader can find more information on interpolation in \cite{InterpolationSpaces}.

This brings us to consider the following definition due to Pisier \cite{Pisier}: a Banach space $\B$ is called \textit{strictly $\theta$-Hilbertian} for $0 < \theta \leq 1$, if there is a compatible pair $(\B_0,\B_1)$ with $\B_1$ a Hilbert space such that $\B = [\B_0, \B_1]_\theta$. Examples of strictly $\theta$-Hilbertian spaces are $L^p$ space and non-commutative $L^p$ spaces (see \cite{PX} for definitions and properties of non-commutative $L^p$ spaces), where in these cases $\theta = \frac{2}{p}$ if $2 \leq  p  < \infty $ and $\theta = 2 - \frac{2}{p}$ if $1 < p \leq 2$. 

For our use, it will be important to bound the norm of an operator of the form $T \otimes \id_{\B}$ given that $\B$ is an interpolation space.

\begin{lemma} \cite[Lemma 3.1]{Salle}
\label{interpolation fact}
Let $(\B_0,\B_1)$ be a compatible pair , $V$ be a finite set, $m: V \rightarrow \mathbb{R}_+$ be a function and $T \in B(\ell^2 (V,m))$ be an operator. Then for every $0 \leq \theta \leq 1$,
$$\Vert T \otimes \id_{[\B_0, \B_1]_\theta} \Vert_{B(\ell^2 (V,m; [\B_0, \B_1]_\theta))} \leq \Vert T \otimes \id_{\B_0} \Vert_{B(\ell^2 (V,m ; \B_0))}^{1-\theta} \Vert T \otimes \id_{\B_1} \Vert_{B(\ell^2 (V,m; \B_1))}^{\theta},$$
where $[\B_0, \B_1]_\theta$ is the interpolation of $\B_0$ and $\B_1$.
\end{lemma}

This Lemma has the following corollary that shows that strictly $\theta$-Hilbertian spaces are uniformly curved (see also \cite[Lemma 3.1]{Salle}):

\begin{corollary}
\label{norm bound on theta-Hil coro}
Let $\B$ be a strictly $\theta$-Hilbertian space with $0 < \theta \leq 1$, $V$ be a finite set, $m: V \rightarrow \mathbb{R}_+$ be a function and $0 < \delta < 1$ be a constant. Assume that $T \in B(\ell^2 (V,m))$ is a fully contractive operator such that $\Vert T \Vert_{B(\ell^2 (V, m))} \leq \delta$. Then $\Vert T \otimes \id_{\B} \Vert_{B(\ell^2 (V, m ; \B))} \leq \delta^{\theta}$.

In other words, if $\B$ is strictly $\theta$-Hilbertian space with $0 < \theta \leq 1$, then for $\alpha (t) = t^{\theta}$, we have that $\B \in \mathcal{E}^{\text{u-curved}}_\alpha$.
\end{corollary}

\begin{proof}
For every Hilbert space $\B_1$ we have that $\Vert T \otimes \id_{\B_1} \Vert_{B(\ell^2 (V, m ; \B_1))} \leq \delta$ and thus the assertion stated above follows from Lemma \ref{interpolation fact}.
\end{proof}

\begin{corollary}
\label{theta-hil is in uni-curv cor}
For a constant $0 < \theta_0 \leq 1$, denote $\mathcal{E}_{\theta_0}$ to be the smallest class of Banach spaces that contains all strictly $\theta$-Hilbertian Banach spaces for all $\theta_0 \leq \theta \leq 1$ and is closed under subspaces, quotients, $\ell^2$-sums and ultraproducts of Banach spaces. Then for every  $0 < \theta_0 \leq 1$, we have that $\mathcal{E}_{\theta_0} \subseteq \mathcal{E}^{\text{u-curved}}_{\alpha (t) = t^{\theta_0}}$. 
\end{corollary}

\begin{remark}
A deep result of Pisier shows that the converse of the above Corollary is ``almost true'' if one considers arcwise $\theta_0$-Hilbertian spaces (see definition in \cite[Section 6]{Pisier2}). Namely, by \cite[Corollary 6.7]{Pisier2}, for every $\theta_0 < \theta \leq 1$ it holds that every Banach space in $\mathcal{E}^{\text{u-curved}}_{\alpha (t) = t^{\theta}}$ is a subquotient of an arcwise $\theta_0$-Hilbertian space. We will not define arcwise $\theta_0$-Hilbertian spaces here and we will make no use of this fact.
\end{remark}

\subsection{Random walks on finite graphs}
\label{rw on graph subsec}
Given a finite graph $(V,E)$, a weight function on $(V,E)$ is a function $m: E \rightarrow \mathbb{R}_+$ and $(V,E)$ with a weight function is called a weighted graph. Given a weighted graph as above, we define for every $v \in V$, $m(v) = \sum_{e \in E, v \in e} m(e)$ and $m(\emptyset) = \sum_{v \in V} m(v)$. 

We also define $\ell^2 (V,m)$ as in \cref{Vector valued subsec} above, i.e., $\ell^2 (V,m)$ is the space of functions $\phi: V \rightarrow \mathbb{C}$ with an inner-product
$$\langle \phi, \psi \rangle = \sum_{\lbrace v \rbrace \in V} = m(v) \phi (v) \overline{\psi (v)}.$$

The \textit{random walk} on $(V,E)$ as above is the operator $A: \ell^2 (V,m) \rightarrow \ell^2 (V,m)$ defined as  
$$(A \phi) (v) = \sum_{u \in V, \lbrace u, v \rbrace \in E} \frac{m (\lbrace u, v \rbrace)}{m(v)} \phi (u).$$
We state without proof a few basic facts regarding the random walk operator:
\begin{enumerate}
\item With the inner-product defined above, $A$ is a self-adjoint operator and the eigenvalues of $A$ lie in the interval $[-1,1]$.
\item The space of constant functions is an eigenspace of $A$ with eigenvalue $1$ and if $(V,E)$ is connected all other the other eigenfunctions of $A$ have eigenvalues strictly less than $1$.
\item The graph $(V,E)$ is bipartite if and only if $-1$ is an eigenvalue of $A$. 
\end{enumerate}

In the case where $m$ is constant $1$ on all the edges, then for every vertex $v$, $m(v)$ is the valance of $v$ and $A$ is called the \textit{simple random walk} on $(V,E)$. 

We denote $M$ to be the orthogonal projection on the space of constant functions: explicitly, for every $\phi \in \ell^2 (V,m)$, $M\phi$ is the constant function
$$M \phi \equiv \frac{1}{m (\emptyset)} \sum_{v \in V} m(v) \phi (v).$$
We note that by the facts stated above, $AM =M$ and if $(V,E)$ is connected and not bipartite, then $\Vert A (I-M) \Vert <1$, where $\Vert . \Vert$ denotes the operator norm. We recall the following definition of spectral expansion that appeared in the introduction for non-weighted graphs:
\begin{definition}
Let $(V,E)$ be a finite connected graph with a weight function $m$ and $0 \leq \lambda <1$ be a constant. The graph $(V,E)$ is called a one-sided $\lambda$-spectral expander if the spectrum of $A (I-M)$ is contained in $[-1, \lambda]$. The graph $(V,E)$ is called a one-sided $\lambda$-spectral expander if the spectrum of $A (I-M)$ is contained in $[- \lambda, \lambda]$ or equivalently if $\Vert A (I-M) \Vert \leq \lambda$.
\end{definition}

Given a Banach space $\B$, we can consider the operator $(A (I-M)) \otimes \id_{\B} : \ell^2 (V,m ; \B) \rightarrow \ell^2 (V,m ; \B)$. 
\begin{claim}
For every graph $(V,E)$ and every Banach space $\B$, $\Vert (A (I-M)) \otimes \id_{\B} \Vert_{B(\ell^2 (V,m ; \B))} \leq 2$.
\end{claim}

\begin{proof}
By triangle inequality and linearity,
\begin{dmath*}
\Vert (A (I-M)) \otimes \id_{\B} \Vert_{B(\ell^2 (V,m ; \B))} \leq \\
\Vert A \otimes \id_{\B} \Vert_{B(\ell^2 (V,m ; \B))} + \Vert A \otimes \id_{\B}  \Vert_{B(\ell^2 (V,m ; \B))} \Vert M \otimes \id_{\B} \Vert_{B(\ell^2 (V,m ; \B))},
\end{dmath*}
and therefore in order to prove the claim, it is enough to show that 
$$\Vert A \otimes \id_{\B} \Vert_{B(\ell^2 (V,m ; \B))} \leq 1, \Vert M \otimes \id_{\B} \Vert_{B(\ell^2 (V,m ; \B))} \leq 1.$$
Indeed, by the convexity of the function $\vert . \vert^2$, for every $\phi \in \ell^2 (V,m ; \B)$, 
\begin{dmath*}
\Vert (A \otimes \id_{\B}) \phi \Vert^2 = \\
\sum_{v \in V} m(v) \vert \sum_{u \in V, \lbrace u, v \rbrace \in E} \frac{m (\lbrace u, v \rbrace)}{m(v)} \phi (u) \vert^2 \leq \\
\sum_{v \in V} m(v) \sum_{u \in V, \lbrace u, v \rbrace \in E} \frac{m (\lbrace u, v \rbrace)}{m(v)} \vert \phi (u) \vert^2 = \\
\sum_{u \in V} \vert \phi (u) \vert^2 \sum_{v \in V, \lbrace u, v \rbrace \in E} m (\lbrace u, v \rbrace) = \\
\sum_{u \in V} m(u) \vert \phi (u) \vert^2  = \Vert \phi \Vert^2,
\end{dmath*}
and
\begin{dmath*}
\Vert (M \otimes \id_{\B}) \phi \Vert^2 = \\
\sum_{v \in V} m(v) \vert \frac{1}{m (\emptyset)} \sum_{u \in V} m(u) \phi (u) \vert^2 \leq 
\sum_{v \in V}  \sum_{u \in V} \frac{m(u)}{m (\emptyset)}   \vert \phi (u) \vert^2 =  \\
\sum_{u \in V}  m(u)  \vert \phi (u) \vert^2 \sum_{v \in V} \frac{m(v)}{m (\emptyset)} = 
\sum_{u \in V}  m(u)  \vert \phi (u) \vert^2 = \Vert \phi \Vert^2.
\end{dmath*}
\end{proof}

Combining this claim with Lemma \ref{L contractive prop} and Corollary \ref{norm bound on theta-Hil coro} yields:
\begin{corollary}
\label{norm bound on rw on uc coro}
Let $(V,E)$ be a connected finite graph with a weight function $m$ and  $0 < \lambda <1$ be a constant such that $(V,E)$ is a two-sided $\lambda$-spectral expander. For every monotone increasing function  $\alpha : (0,1] \rightarrow (0,1]$ such that $\lim_{t \rightarrow 0^{+}} \alpha (t) = 0$ and every $\B \in \mathcal{E}^{\text{u-curved}}_\alpha$, we have that 
$$\Vert (A (I-M)) \otimes \id_{\B} \Vert_{B(\ell^2 (V,m ; \B))} \leq 2 \alpha (\lambda).$$

In particular, for every $0 < \theta \leq 1$ and every strictly $\theta$-Hilbertian space $\B$, we have that 
$$\Vert (A (I-M)) \otimes \id_{\B} \Vert_{B(\ell^2 (V,m ; \B))} \leq 2 \lambda^\theta.$$
\end{corollary}

\subsection{Weighted simplicial complexes}

Let $X$ be an $n$-dimensional simplicial complex. For $-1 \leq k \leq n$, we denote $X(k)$ to be the $k$-dimensional faces of $X$ and $X = \bigcup_{k} X(k)$. $X$ is called \textit{pure} $n$-dimensional if for every $\tau$ in $X$, there is $\sigma \in X(n)$ such that $\tau \subseteq \sigma$. $X$ is called \textit{locally finite} if for every $\lbrace v \rbrace \in X (0)$, $\vert \lbrace \sigma \in X(n) : v \in \sigma \rbrace \vert < \infty$. Throughout this paper, we will always assume that $X$ is pure $n$-dimensional and locally finite.  

We define the following weight function $m: \bigcup_{k=0}^{n} \X (k) \rightarrow \mathbb{R}$ inductively as follows: 
$$\forall \sigma \in \X (n), m(\sigma) =1,$$
For $0 \leq k \leq n-1$ and $\tau \in X(k)$,
$$m (\tau) = \sum_{\sigma \in \X (k+1), \tau \subseteq \sigma} m(\sigma).$$
More explicitly,
$$\forall \tau \in \X (k), m(\tau) = (n-k)! \vert \lbrace \sigma \in \X (n) : \tau \subseteq \sigma  \rbrace \vert.$$
In the case where $X$ is finite, we also define $m(\emptyset) = \sum_{\lbrace v \rbrace \in X(0)} m(\lbrace v \rbrace)$. 

Given a simplex $\tau \in X(j)$, the \textit{link of $\tau$} is the subcomplex of $X$, denoted $X_\tau$, that is defined as
$$X_\tau = \lbrace \eta \in X : \tau \cap \eta = \emptyset, \tau \cup \eta \in X \rbrace.$$
We note that by the assumption that $X$ is locally finite, it follows that $X_\tau$ is finite and by the assumption that $X$ is pure $n$-dimensional, it follows that $X_\tau$ is pure $(n-j-1)$-dimensional (where $j$ is the dimension of $\tau$). The weight function on $X_\tau$, denoted by $m_\tau$ is defined as above:
$$\forall \sigma \in \X_\tau (n-j-1), m_\tau (\sigma) =1,$$
For $0 \leq k \leq (n-j-1)-1$ and for $\eta \in X(k)$,
$$m_\tau (\eta) = \sum_{\sigma \in \X_\tau (k+1), \eta \subseteq \sigma} m_\tau (\sigma).$$
We observe that $m_\tau (\eta) = m (\tau \cup \eta)$: indeed, if $\eta \in \X_\tau (n-j-1)$, then $\tau \cup \eta \in X(n)$ and therefore
$$m_\tau (\eta) = 1 = m (\tau \cup \eta).$$
For $0 \leq k \leq (n-j-1)-1$ and $\eta \in X(k)$, the equality follows by induction:
\begin{dmath*}
m_\tau (\eta) = \sum_{\sigma \in \X_\tau (k+1), \eta \subseteq \sigma} m_\tau (\sigma) =  
\sum_{\sigma \in \X_\tau (k+1), \eta \subseteq \sigma} m (\tau \cup \sigma) = 
\sum_{\tau \cup \sigma \in \X ((j+1)+k+1), \eta \subseteq \tau \cup \sigma} m (\tau \cup \sigma) = m(\eta).
\end{dmath*}

\subsection{Group representations on Banach spaces} 
Let $\Gamma$ be a locally compact group and $\B$ a Banach space. Let $\pi$ be a representation $\pi :  \Gamma \rightarrow B (\B)$, where $B(\B)$ are the bounded linear operators on $\B$. Throughout this paper we shall always assume $\pi$ is continuous with respect to the strong operator topology without explicitly mentioning it. We recall that given $\pi$, the dual representation $\overline{\pi} : \Gamma \rightarrow B (\B^*)$ is defined as 
$$\langle x, \overline{\pi} (g). y  \rangle =  \langle \pi (g^{-1}). x, y  \rangle, \forall g \in \Gamma, x \in \B, y \in \B^*.$$

Observe that if $\pi$ is an isometric representation, then $\overline{\pi}$ is an isometric representation: Indeed, for every $g \in \Gamma$,
$$\max_{x \in \B, y \in \B^*, \vert x \vert = \vert y \vert =1} \langle x, \overline{\pi} (g). y  \rangle = \max_{x \in \B, y \in \B^*, \vert x \vert = \vert y \vert =1} \langle \pi (g^{-1}). x, y  \rangle = \vert \pi (g^{-1}). x \vert =  1,$$
i.e., for every $g \in \Gamma$ and every $y \in \B^*$, if $\vert y \vert = 1$, then $\vert \overline{\pi} (g) y \vert =1$ and it follows that $\overline{\pi}$ is isometric.

We remark that $\overline{\pi}$ might not be continuous for a general Banach space, but it is continuous for a large class of Banach spaces, called Asplund spaces:

\begin{definition}
A Banach space $\B$ is said to be an Asplund space if every separable subspace of $\B$ has a separable dual.
\end{definition}

There are many examples of Asplund spaces and in particular every reflexive space is Asplund (see \cite{Yost} for an exposition on Asplund spaces). The reason we are interested in Asplund spaces is the following theorem of Megrelishvili:

\begin{theorem}\cite[Corollary 6.9]{Megre}
\label{Asplund implies continuous dual rep}
Let $\Gamma$ be a topological group and let $\pi$ be a continuous representation of $\Gamma$ on a Banach space $\B$. If $\B$ is an Asplund space, then the dual representation $\overline{\pi}$ is also continuous. In particular, if $\B$ is reflexive, then the dual representation $\overline{\pi}$ is continuous.
\end{theorem}

\section{Equivariant cohomology}
\label{equiv coho sec}
Let $X$ be a locally finite, pure $n$-dimensional simplicial complex with the weight function $m$ defined above and $\Gamma$ be a locally compact, unimodular group acting cocompactly and properly on $X$. Also let $\B$ be a reflexive Banach space and $\pi$ be a continuous isometric representation. 

\begin{remark}
\label{cont remark}
By our assumption, $\B$ is reflexive and thus Asplund. Therefore, by Theorem \ref{Asplund implies continuous dual rep}, the assumption of continuity of $\pi$ implies that $\overline{\pi}$ is also continuous.
\end{remark}

Below, we will define the equivariant cohomology $H^k (X, \pi)$ and prove a general criterion for the vanishing of this cohomology. All the definitions below regarding cohomology already appeared in \cite{BS} for representations on Hilbert spaces and were generalized to the Banach setting in \cite{Koivisto}. The criterion for vanishing of cohomology appeared (in a somewhat different form) in \cite{Nowak} (and also in \cite{Koivisto}) and we claim no originality here. 

In order to define the equivariant cohomology, we introduce the following notation (based on \cite{BS}):
\begin{enumerate}
\item For $0\leq k\leq n$, denote by $\Sigma(k)$ the set of ordered $k$-simplices (i.e. $\sigma \in \Sigma(k)$ is and ordered $(k+1)$-tuple of vertices). 
\item For a map $\phi : \Sigma (k) \rightarrow \B$, $\phi$ is called \textit{alternating} if for every permutation $\gamma \in \Sym \lbrace 0,...,k \rbrace$ and every $(v_0,...,v_k) \in \Sigma (k)$,
$$\phi ((v_{\gamma (0)},...,v_{\gamma (k)})) = \sgn (\gamma) \phi ((v_0,...,v_k)).$$
Also, $\phi$ is called \textit{equivariant} if for every $g \in \Gamma$ and every $\sigma \in\Sigma(k)$,
$$\pi(g)\phi(\sigma)=\phi(g. \sigma).$$
\item For $0\leq k\leq n$, a $k$-cochain twisted by $\pi$ is a map $\phi : \Sigma (k) \rightarrow \B$ that is both alternating and equivariant. We denote $C^{k}(X,\pi)$ to be the space of all $k$-cochains twisted by $\pi$.  
\end{enumerate}

For $0 \leq k <n$, the \emph{differential} $d_k :C^{k}(X,\pi)\rightarrow C^{k+1}(X,\pi)$
is given by \[
d_k \phi(\sigma):=\sum_{i=0}^{k+1}(-1)^{i}\phi(\sigma_{i}),\;\sigma\in\Sigma(k+1)\]
where $\sigma_{i}=\left(v_{0},...,\hat{v_{i}},...,v_{k+1}\right)$
for $\left(v_{0},...,v_{k+1}\right)=\sigma\in\Sigma(k+1)$. By a standard computation $d_k \circ d_{k-1} =0$ and we define the $k$-th cohomology as $ H^k (X, \pi ) = \Ker (d_k) / \im (d_{k-1})$.

\begin{remark}
The reader should note that in the definition of the cohomology above, we made no use of the fact that $\B$ is a Banach space, and this definition apply in a much more general setting.
\end{remark}

We define a norm on $C^{k}(X,\pi)$ in order to make it into a Banach space:
\begin{enumerate}
\item We choose a set, denoted $\Sigma(k,\Gamma)\subseteq\Sigma(k)$, of representatives for the action of $\Gamma$ on $\Sigma (k)$. We note that by the equivariance assumption, $\phi \in C^k (X, \pi)$ is determined by its values on $\Sigma(k,\Gamma)$. We also note that by the assumption that the action of $\Gamma$ is cocompact, $\Sigma(k,\Gamma)$ is a finite set.
\item We extend the weight function $m$ defined above to ordered simplices, by forgetting the ordering, i.e., for every $(v_0,...,v_k) \in \Sigma (k)$, we define $m((v_0,...,v_k)) = m(\lbrace v_0,...,v_k \rbrace)$.
\item For a simplex $\sigma\in\Sigma(k)$, we denote $\Gamma_{\sigma}$ to be the point-wise stabilizer of $\sigma$, i.e., for $\sigma = (v_0,...,v_k)$, then $g \in \Gamma_{\sigma}$ if and only if for every $0 \leq i \leq k$ it holds that $g. v_i = v_i$.  We further denote $\vert \Gamma_{\sigma} \vert$ to be the measure of $\Gamma_{\sigma}$ with respect to the Haar measure of $\Gamma$. By the assumption that the action of $\Gamma$ is proper, it follows that $\vert \Gamma_{\sigma} \vert < \infty$.
\item We define a norm on $C^{k}(X,\pi)$ by
$$\Vert \phi \Vert = \left( \sum_{\sigma \in \Sigma (k, \Gamma)} \frac{m(\sigma)}{(k+1)!\left|\Gamma_{\sigma}\right|} \vert \phi (\sigma) \vert^2 \right)^{\frac{1}{2}},$$
where $\vert . \vert$ denotes the norm of $\B$.
\end{enumerate}
With the definitions above, $C^{k}(X,\pi)$ is a normed space and we leave it to the reader to verify that it is a Banach space (this is almost immediate due to (1) above). 

\begin{proposition}
\label{C^k is reflexive prop}
The space $C^{k}(X,\pi)$ is reflexive.
\end{proposition}

\begin{proof}
Define $\B^{\Sigma(k,\Gamma)} = \lbrace \phi : \Sigma(k,\Gamma)  \rightarrow \B \rbrace$ with the norm 
$$\Vert \phi \Vert = \left( \sum_{\sigma \in \Sigma (k, \Gamma)} \frac{m(\sigma)}{(k+1)!\left|\Gamma_{\sigma}\right|} \vert \phi (\sigma) \vert^2 \right)^{\frac{1}{2}}.$$
This is a reflexive Banach space, since it is a weighted $\ell^2$ sum of $\vert \Sigma(k,\Gamma) \vert$ copies of $\B$. We note that $C^k (X, \pi)$ is a closed subspace of $\B^{\Sigma(k,\Gamma)}$ and thus it is also reflexive.
\end{proof}

Choose $\Sigma ' (k, \Gamma) \subseteq \Sigma (k, \Gamma)$ to be a set of representatives of the action of the permutation group $\Sym \lbrace 0,...,k \rbrace$ on $\Sigma (k, \Gamma)$, i.e., for every $(v_0,...,v_k) \in \Sigma (k, \Gamma)$ there is a unique permutation $\gamma \in \Sym \lbrace 0,...,k \rbrace$ such that $(v_{\gamma (0)},...,v_{\gamma (k)}) \in \Sigma ' (k, \Gamma)$. By definition all the cochains in $C^{k}(X,\pi)$ are equivariant and alternating and thus every map in $C^k (X, \pi)$ is uniquely determined by its values on $\Sigma ' (k, \Gamma)$. However, it may be the case that not every map $\phi ' : \Sigma ' (k, \Gamma) \rightarrow \B$ can be extended to an equivariant and alternating map on $\Sigma (k)$. Below, we will give a necessary and sufficient condition for the existence of such extension.

For $\sigma \in \Sigma (k)$, we denote $\Gamma_\sigma^{+}$ and $\Gamma_\sigma^{-}$ to be the subsets of $\Gamma$ that (when restricted to $\sigma$) induce even and odd permutations on $\sigma$, i.e., for $\sigma = (v_0,...,v_k)$
\begin{align*}
\Gamma_\sigma^{+} =  \lbrace g \in \Gamma : g.(v_0,...,v_k) = (v_{\gamma (0)},...,v_{\gamma (k)}), \gamma \in \Sym \lbrace 0,...,k \rbrace \\ \text{ and } \gamma \text{ is an even permutation} \rbrace,
\end{align*}
\begin{align*}
\Gamma_\sigma^{-} =  \lbrace g \in \Gamma : g.(v_0,...,v_k) = (v_{\gamma (0)},...,v_{\gamma (k)}), \gamma \in \Sym \lbrace 0,...,k \rbrace \\ \text{ and } \gamma \text{ is an odd permutation} \rbrace.
\end{align*}
We note that $\Gamma_\sigma^+$ is a subgroup of $\Gamma$ and that 
$$\forall g \in \Gamma_\sigma^{+}, g.\Gamma_\sigma^{+} = \Gamma_\sigma^{+}, g.\Gamma_\sigma^{-} = \Gamma_\sigma^{-},$$
$$\forall g \in \Gamma_\sigma^{-}, g.\Gamma_\sigma^{+} = \Gamma_\sigma^{-}, g.\Gamma_\sigma^{-} = \Gamma_\sigma^{+}.$$

Define the subspace $\B_{\sigma, \pi} \subseteq \B$ to be the subspace of vectors $x \in \B$ such that 
$$\forall g \in \Gamma_\sigma^{+}, \pi (g).x = x,$$
and
$$\forall g \in \Gamma_\sigma^{-}, \pi (g).x = -x.$$

\begin{proposition}
\label{B_sigma,pi prop}
A map $\phi ' : \Sigma ' (k, \Gamma) \rightarrow \B$ can be extended (uniquely) to a map $\phi \in C^k (X, \pi)$ if and only if for every $\sigma \in \Sigma ' (k, \Gamma)$ it holds that $\phi ' (\sigma) \in \B_{\sigma, \pi}$.  
\end{proposition}

\begin{proof}
Let $\phi ' : \Sigma ' (k, \Gamma) \rightarrow \B$ be some map. 

Assume first that there is a map $\phi \in C^k (X, \pi)$ such that $\left. \phi \right\vert_{\Sigma ' (k, \Gamma)} = \phi '$. Let $(v_0,...,v_k) \in \Sigma ' (k, \Gamma)$ and $g \in \Gamma_\sigma^{+}$. Also let $\gamma \in \Sym \lbrace 0,...,k \rbrace$ such that $\gamma$ is even and 
$g.(v_0,...,v_k) = (v_{\gamma (0)},...,v_{\gamma (k)})$. Then it holds that
\begin{dmath*}
\pi (g). \phi ' ((v_0,...,v_k)) = \pi (g). \phi ((v_0,...,v_k)) =^{\phi \text{ is equivariant}} 
 \phi (g.(v_0,...,v_k)) = \phi ((v_{\gamma (0)},...,v_{\gamma (k)})) =^{\phi \text{ is alternating}} 
  \phi ((v_0,...,v_k)) = \phi ' ((v_0,...,v_k)),
\end{dmath*}
i.e., for every $\sigma \in \Sigma ' (k, \Gamma)$ and every $g \in \Gamma_\sigma^{+}$, $\pi (g). \phi ' (\sigma) = \phi ' (\sigma)$. By a similar computation, it follows that for every $\sigma \in \Sigma ' (k, \Gamma)$ and every $g \in \Gamma_\sigma^{-}$, $\pi (g). \phi ' (\sigma) = -\phi ' (\sigma)$. Thus, for every $\sigma \in \Sigma ' (k, \Gamma)$ it holds that $\phi ' (\sigma) \in \B_{\sigma, \pi}$.

In the other direction, assume that for every $\sigma \in \Sigma ' (k, \Gamma)$ it holds that $\phi ' (\sigma) \in \B_{\sigma, \pi}$. For every $\gamma \in \Sym \lbrace 0,...,k \rbrace$, every $g \in \Gamma$ and every $(v_0,...,v_k) \in \Sigma ' (k, \Gamma)$, define
$$\phi (g. (v_{\gamma (0)},...,v_{\gamma (k)})) = \pi (g) \sgn (\gamma) \phi ' ((v_0,...,v_k)).$$
If we show that $\phi$ above is well defined, it will follow from its definition that it is equivariant and alternating. Fix  $\sigma = (v_0,...,v_k) \in \Sigma ' (k, \Gamma)$ and let $\gamma, \gamma ' \in \Sym \lbrace 0,...,k \rbrace, g, g' \in \Gamma$ be such that 
$$ g. (v_{\gamma (0)},...,v_{\gamma (k)}) = g '. (v_{\gamma ' (0)},...,v_{\gamma ' (k)}).$$
Then 
$$  (v_{\gamma (\gamma ')^{-1} (0)},...,v_{\gamma (\gamma ')^{-1}(k)}) = (g^{-1} g '). (v_{0},...,v_{k})$$
and therefore $g^{-1} g ' \in \Gamma_\sigma^+ \cup \Gamma_\sigma^-$ and the sign of the permutation induced by $g^{-1} g '$ on $\sigma$ is exactly $\sgn (\gamma (\gamma ')^{-1}) = \sgn (\gamma) \sgn (\gamma ')$. From the assumption that $\phi ' ((v_0,...,v_k)) \in \B_{\sigma, \pi}$ it follows that 
$$\pi (g^{-1} g') \phi ' ((v_0,...,v_k)) = \sgn (\gamma) \sgn (\gamma ') \phi ' ((v_0,...,v_k)),$$
or equivalently
$$\sgn (\gamma) \sgn (\gamma ') \pi (g^{-1} g') \phi ' ((v_0,...,v_k))  =  \phi ' ((v_0,...,v_k)),$$
Thus
\begin{dmath*}
\pi (g) \sgn (\gamma) \phi ' ((v_0,...,v_k)) = \pi (g) \sgn (\gamma) \sgn (\gamma) \sgn (\gamma ') \pi (g^{-1} g') \phi ' ((v_0,...,v_k)) = 
\pi (g' )  \sgn (\gamma ') \phi ' ((v_0,...,v_k)),
\end{dmath*}
and $\phi$ is well-defined.
\end{proof}

All the results above were stated for $\pi$, but since $\overline{\pi}$ is a representation of $\Gamma$ on a reflexive Banach space they pass automatically to $\overline{\pi}$, i.e., we can define $C^k (X,\overline{\pi})$ as above and by the same considerations it follows that $C^k (X,\overline{\pi})$ is also a reflexive Banach space. We also denote $\overline{d_k}: C^{k}(X,\overline{\pi})\rightarrow C^{k+1}(X,\overline{\pi})$ to be the differential defined as above. 
 
The reason for considering $C^{k}(X,\overline{\pi})$ is that there is a natural coupling between $C^k (X,\pi)$ and $C^{k}(X,\overline{\pi})$:  Let $(.,.)$ denote the usual coupling between $\B$ and $\B^*$ and for $\phi \in C^k (X, \pi), \psi \in C^k (X,\overline{\pi})$ define
\[
\left\langle \phi,\psi\right\rangle :=\sum_{\sigma\in\Sigma(k,\Gamma)}\frac{m(\sigma)}{(k+1)!\left|\Gamma_{\sigma}\right|} ( \phi(\sigma),\psi(\sigma)) . \]

With the above coupling, $C^k (X, \overline{\pi}) \subseteq (C^k (X, \pi))^*$. Actually, since $\B$ is reflexive, there is an isomorphism between $C^k (X, \overline{\pi})$ and  $(C^k (X, \pi))^*$ (see \cite[Proposition 28]{Koivisto}), but we will make no use of this fact. Given this coupling, we denote $d_k^* :C^{k+1}(X,\overline{\pi})\rightarrow C^{k}(X,\overline{\pi})$ to be the adjoint operator of $d_k$ and $\overline{d_k}^* :C^{k+1}(X,\pi)\rightarrow C^{k}(X,\pi)$ to be the adjoint operator of $\overline{d_k}$.

We recall that for a Banach space $\B$, the duality mapping is a mapping $j : \B \rightarrow 2^{\B^*}$ defined as
$$j (x) = \lbrace x^* \in \B^* : \vert x \vert = \vert x^* \vert, (x,x^*) = \vert x \vert^2 \rbrace,$$
(the fact that the set defined by $j(x)$ is non-empty follows immediately from Hahn-Banach). By our assumption, $\B$ is reflexive and thus we also have the duality mapping $\overline{j} : \B^* \rightarrow 2^{\B} (= 2^{\B^{**}})$.

We define maps $J: C^k (X, \pi) \rightarrow 2^{C^k (X, \overline{\pi})}$ and $\overline{J}: C^k (X, \overline{\pi}) \rightarrow 2^{C^k (X, \pi)}$ by 
$$\forall \phi \in C^k (X, \pi), J \phi = \lbrace \psi \in C^k (X, \overline{\pi}) : \forall \sigma \in \Sigma (k), \psi (\sigma) \in j(\phi (\sigma)) \rbrace, $$
$$\psi \in C^k (X, \overline{\pi}), \overline{J} \psi = \lbrace \phi \in C^k (X, \pi) : \forall \sigma \in \Sigma (k), \phi (\sigma) \in \overline{j} (\psi (\sigma)) \rbrace. $$
\begin{proposition}
\label{J prop}
Let $X, \B, \pi$, $J, \overline{J}$ be as above and $\phi \in C^k (X, \pi)$, $\psi \in C^k (X, \overline{\pi})$. Then $J \phi$, $\overline{J} \psi$ are non empty sets and  
$$\forall \phi^* \in J \phi, \Vert \phi^* \Vert^2 = \Vert \phi \Vert^2 = \langle \phi, \phi^* \rangle,$$ 
$$\forall \psi^* \in \overline{J} \psi, \Vert \psi^* \Vert^2 = \Vert \psi \Vert^2 = \langle \psi^* , \psi \rangle.$$
\end{proposition}

\begin{proof}
We will prove the assertions above only for $J \phi$, since the proof for $\overline{J} \psi$ is similar.  

We will only show that $J \phi$ is non-empty: the fact that for every $\phi^* \in J \phi$, 
$$ \Vert \phi^* \Vert^2 = \Vert \phi \Vert^2 = \langle \phi, \phi^* \rangle,$$
follows from straight-forward a computation that is left for to the reader.

Fix $\phi \in C^k (X, \pi)$. Choose $\Sigma ' (k, \Gamma) \subseteq \Sigma (k, \Gamma)$ as above to be a set of representatives of the action of the permutation group $\Sym \lbrace 0,...,k \rbrace$ on $\Sigma (k, \Gamma)$. By Proposition \ref{B_sigma,pi prop}, it is enough to show that there is $\psi ' : \Sigma ' (k, \Gamma) \rightarrow \B^*$ such that for every $\sigma \in \Sigma ' (k, \Gamma)$ it holds that $\psi ' (\sigma) \in \B^{*}_{\sigma, \overline{\pi}}$ and $\psi ' (\sigma) \in j (\phi (\sigma))$.  

For every $\sigma \in \Sigma ' (k, \Gamma)$, we choose some $x_\sigma^* \in j(\phi (\sigma))$ and define
$$\psi ' (\sigma) = \frac{1}{2 \vert \Gamma_\sigma^+ \vert} \int_{\Gamma_\sigma^+} \overline{\pi} (g). x_\sigma^* d \mu (g) -  \frac{1}{2 \vert \Gamma_\sigma^- \vert} \int_{ \Gamma_\sigma^-} \overline{\pi} (g). x_\sigma^* d \mu (g),$$
(if $\Gamma_\sigma^-$ is an empty set, then the second integral is omitted). These integrals is well defined because by our assumptions the action of $\overline{\pi}$ is continuous and $\Gamma_\sigma^+,  \Gamma_\sigma^-$ are compact sets. 

Recall that for every $g ' \in \Gamma_\sigma^+$ it holds that $g '. \Gamma_\sigma^+ =\Gamma_\sigma^+, g '. \Gamma_\sigma^- =\Gamma_\sigma^-$ and that the action of $\Gamma$ preserves the Haar measure. Thus for every $g ' \in \Gamma_\sigma^+$ and every $\sigma \in \Sigma ' (k, \Gamma)$ it holds that 
\begin{dmath*}
\overline{\pi} (g '). \psi ' (\sigma) = 
 \frac{1}{2 \vert \Gamma_\sigma^+ \vert} \int_{\Gamma_\sigma^+} \overline{\pi} (g ' g). x_\sigma^* d \mu (g) -  \frac{1}{2 \vert \Gamma_\sigma^- \vert} \int_{ \Gamma_\sigma^-} \overline{\pi} (g ' g). x_\sigma^* d \mu (g) =^{g'' = g' g} 
  \frac{1}{2 \vert \Gamma_\sigma^+ \vert} \int_{g ' . \Gamma_\sigma^+} \overline{\pi} (g ''). x_\sigma^* d \mu (g '') -  \frac{1}{2 \vert \Gamma_\sigma^- \vert} \int_{ g ' . \Gamma_\sigma^-} \overline{\pi} (g ''). x_\sigma^* d \mu (g '') = 
\frac{1}{2 \vert \Gamma_\sigma^+ \vert} \int_{\Gamma_\sigma^+} \overline{\pi} (g ''). x_\sigma^* d \mu (g '') -  \frac{1}{2 \vert \Gamma_\sigma^- \vert} \int_{ \Gamma_\sigma^-} \overline{\pi} (g ''). x_\sigma^* d \mu (g '') = 
\psi ' (\sigma).
\end{dmath*}

Similarly, since for every $g' \in \Gamma_\sigma^-$ it holds that $g '. \Gamma_\sigma^+ =\Gamma_\sigma^-, g '. \Gamma_\sigma^+ =\Gamma_\sigma^+$
it follows that for every $g' \in  \Gamma_\sigma^-$ and every $\sigma \in \Sigma ' (k, \Gamma)$,
$$\overline{\pi} (g '). \psi ' (\sigma) = -  \psi ' (\sigma)$$
and thus $\psi ' (\sigma) \in \B^{*}_{\sigma, \overline{\pi}}$.

We note that for every $\sigma \in \Sigma ' (k, \Gamma)$ it holds that 
\begin{dmath*}
(\phi (\sigma), \psi ' (\sigma)) =  
(\phi (\sigma), \frac{1}{2 \vert \Gamma_\sigma^+ \vert} \int_{\Gamma_\sigma^+} \overline{\pi} (g). x_\sigma^* d \mu (g) -  \frac{1}{2 \vert \Gamma_\sigma^- \vert} \int_{ \Gamma_\sigma^-} \overline{\pi} (g). x_\sigma^* d \mu (g)) = 
\frac{1}{2 \vert \Gamma_\sigma^+ \vert} \int_{\Gamma_\sigma^+} (\phi (\sigma),  \overline{\pi} (g). x_\sigma^* ) d \mu (g) - \frac{1}{2 \vert \Gamma_\sigma^- \vert} \int_{\Gamma_\sigma^-} (\phi (\sigma),  \overline{\pi} (g). x_\sigma^* ) d \mu (g)= 
\frac{1}{2 \vert \Gamma_\sigma^+ \vert} \int_{\Gamma_\sigma^+} (\pi (g^{-1}).\phi (\sigma),  x_\sigma^* ) d \mu (g) - \frac{1}{2 \vert \Gamma_\sigma^- \vert} \int_{\Gamma_\sigma^-} (\pi (g^{-1}). \phi (\sigma), x_\sigma^* ) d \mu (g)= 
\frac{1}{2 \vert \Gamma_\sigma^+ \vert} \int_{\Gamma_\sigma^+} (\phi (\sigma),  x_\sigma^* ) d \mu (g) - \frac{1}{2 \vert \Gamma_\sigma^- \vert} \int_{\Gamma_\sigma^-} (- \phi (\sigma), x_\sigma^* ) d \mu (g)= 
\frac{1}{2 \vert \Gamma_\sigma^+ \vert} \int_{\Gamma_\sigma^+} \vert \phi (\sigma) \vert^2 d \mu (g) + \frac{1}{2 \vert \Gamma_\sigma^- \vert} \int_{\Gamma_\sigma^-} \vert \phi (\sigma) \vert^2 d \mu (g)= 
 \vert \phi (\sigma) \vert^2,
\end{dmath*}
and that 
\begin{dmath*}
\vert \psi ' (\sigma) \vert = 
\left\vert\frac{1}{2 \vert \Gamma_\sigma^+ \vert} \int_{\Gamma_\sigma^+} \overline{\pi} (g). x_\sigma^* d \mu (g) -  \frac{1}{2 \vert \Gamma_\sigma^- \vert} \int_{ \Gamma_\sigma^-} \overline{\pi} (g). x_\sigma^* d \mu (g) \right\vert \leq \\ 
\frac{1}{2 \vert \Gamma_\sigma^+ \vert} \int_{\Gamma_\sigma^+} \vert \overline{\pi} (g). x_\sigma^* \vert d \mu (g) +  \frac{1}{2 \vert \Gamma_\sigma^- \vert} \int_{ \Gamma_\sigma^-} \vert \overline{\pi} (g). x_\sigma^* \vert d \mu (g) = \\
\frac{1}{2 \vert \Gamma_\sigma^+ \vert} \int_{\Gamma_\sigma^+} \vert \phi (\sigma) \vert d \mu (g) +  \frac{1}{2 \vert \Gamma_\sigma^- \vert} \int_{ \Gamma_\sigma^-} \vert \phi (\sigma)  \vert d \mu (g) = \vert \phi (\sigma) \vert
\end{dmath*}
and therefore $\psi ' (\sigma) \in j(\phi (\sigma))$ as needed.
\end{proof}

Below, we will make use of changing the order of summation when calculating norms of maps $C^k(X,\pi)$ or coupling between maps of $C^k(X,\pi)$ and $C^k(X,\overline{\pi})$. For this, we will need the following: for $0\leq l<k\leq n$ and $\tau \in \Sigma (l)$, $\sigma \in \Sigma (k)$, we denote $\tau \subseteq \sigma$ if $\sigma$ contains $\tau$ as a set (without respecting the ordering), i.e., for $\sigma = (v_0,...,v_k), \tau = (w_0,...,w_l)$, $\tau \subseteq \sigma$ if $\lbrace w_0,...,w_l \rbrace \subseteq \lbrace v_0,...,v_k \rbrace$. The following technical proposition is taken from \cite{BS}, \cite{DJ1}:

\begin{proposition}
\label{changing the order of sum prop}
\cite[Lemma 1.3]{BS}, \cite[Lemma 3.3]{DJ1} For $0\leq l<k\leq n$, let $f=f(\tau,\sigma)$ be
a $\Gamma$-invariant function on the set of pairs $\left(\tau,\sigma\right)$,
where $\tau \in \Sigma (l)$, $\sigma \in \Sigma (k)$ with $\tau\subseteq \sigma$. Then \[
\sum_{\sigma\in\Sigma(k,\Gamma)}\sum_{\begin{array}{c}
{\scriptstyle \tau\in\Sigma(l)}\\
{\scriptstyle \tau\subseteq \sigma}\end{array}}\frac{f(\tau,\sigma)}{\left|\Gamma_{\sigma}\right|}=\sum_{\tau\in\Sigma(l,\Gamma)}\sum_{\begin{array}{c}
{\scriptstyle \sigma\in\Sigma(k)}\\
{\scriptstyle \tau\subseteq\sigma}\end{array}}\frac{f(\tau,\sigma)}{\left|\Gamma_{\tau}\right|}\]
\end{proposition}

The reader should note, that from now on we will use the above Proposition to change the order of summation without mentioning it explicitly.

\begin{proposition}
\label{bound on d, calc of d^* prop}
\begin{enumerate}
\item (equivalent to \cite[Proposition 1.5]{BS}) The differential it is a bounded operator and $\Vert d_k \Vert \leq \sqrt{k+2}$.
\item (equivalent to \cite[Proposition 1.6]{BS}) We denote $d_k^* :C^{k+1}(X,\overline{\pi})\rightarrow C^{k}(X,\overline{\pi})$ to be the adjoint operator of $d_k$. Then 
 \[
d_k^* \phi(\tau)=\sum_{ v \in\Sigma(0), v \tau\in \Sigma (k+1)}\frac{m(v\tau)}{m(\tau)}\phi(v\tau),\;\tau\in\Sigma(k)\]
where $v\tau=(v,v_{0},...,v_{k})$ for $\tau=(v_{0},...,v_{k})$.
\end{enumerate}
\end{proposition}

\begin{proof}
\begin{enumerate}
\item For every $\phi \in C^k (X, \pi)$ we have
\begin{dmath*}
 \Vert d_k \phi \Vert^2 = \sum_{\sigma \in \Sigma (k+1, \Gamma)} \dfrac{m(\sigma )}{(k+2)! \vert \Gamma_\sigma \vert} \vert \sum_{i=0}^{k+1}(-1)^{i}\phi(\sigma_{i}) \vert^2 \leq \\
 \sum_{\sigma \in \Sigma (k+1, \Gamma)} \dfrac{m(\sigma )}{(k+2)! \vert \Gamma_\sigma \vert} (k+2) \sum_{i=0}^{k+1} \vert \phi(\sigma_{i}) \vert^2 = \\
 \sum_{\sigma \in \Sigma (k+1, \Gamma)} \dfrac{ m(\sigma )}{(k+1)! (k+1)! \vert \Gamma_\sigma \vert}  \sum_{ \tau \in \Sigma (k), \tau \subset \sigma} \vert \phi(\tau) \vert^2 = \\
 \sum_{\tau \in \Sigma (k, \Gamma)} \dfrac{\vert \phi(\tau) \vert^2 }{(k+1)! (k+1)! \vert \Gamma_\tau \vert}  \sum_{ \sigma \in \Sigma (k+1), \tau \subset \sigma} m(\sigma ) = \\
  \sum_{\tau \in \Sigma (k, \Gamma)} \dfrac{(k+2)! m (\tau ) \vert \phi(\tau) \vert^2 }{ (k+1)!(k+1)! \vert \Gamma_\tau \vert}  = (k+2)  \Vert \phi \Vert^2.
 \end{dmath*}

\item For $\sigma \in \Sigma (k+1)$ and $\tau \subset \sigma, \tau \in \Sigma (k)$ denote by $[\sigma : \tau ]$ the incidence coefficient of $\tau$ with respect to $\sigma$, i.e., if $\sigma_i$ has the same vertices as $\tau$ then for every $\psi \in C^k (X, \pi)$ we have $[ \sigma : \tau ] \psi (\tau ) = (-1)^i \psi (\sigma_i)$. Take $\phi \in C^{k+1} (X, \overline{\pi})$ and $\psi \in C^k (X, \pi)$. We note that for every $\tau \in \Sigma (k)$, every $\sigma \in \Sigma (k+1)$ and every $g \in \Gamma$,
\begin{dmath*}
(\psi (g. \tau), \phi (g. \sigma)) = (\pi (g) \psi (\tau), \overline{\pi} (g) \phi ( \sigma)) = (\psi (\tau), \phi (\sigma)),
\end{dmath*}
and we will use this fact in equality (*) below, in which we apply Proposition \ref{changing the order of sum prop}:
\begin{dmath*}
 \langle d \psi , \phi \rangle = \sum_{\sigma \in \Sigma (k+1, \Gamma)} \dfrac{m(\sigma )}{(k+2)! \vert \Gamma_\sigma \vert} ( \sum_{i=0}^{k+1}(-1)^{i}\psi(\sigma_{i}) , \phi ( \sigma ) ) = 
 {\sum_{\sigma \in \Sigma (k+1, \Gamma)} \dfrac{m(\sigma )}{(k+1)! (k+2)! \vert \Gamma_\sigma \vert} ( \sum_{\tau \in\Sigma(k), \tau \subset \sigma } [\sigma : \tau] \psi(\tau) , \phi ( \sigma ) ) } = 
\sum_{\sigma \in \Sigma (k+1, \Gamma)} \dfrac{m(\tau )}{(k+1)! \vert \Gamma_\sigma \vert} \sum_{\tau \in\Sigma(k), \tau \subset \sigma } (  \psi(\tau) ,\dfrac{[\sigma : \tau] m(\sigma )}{m(\tau ) (k+2)!} \phi ( \sigma ) ) =^{(*)}
 \sum_{\tau \in \Sigma (k, \Gamma)} \dfrac{m(\tau )}{(k+1)! \vert \Gamma_\tau \vert} \sum_{ \sigma \in\Sigma(k+1), \tau \subset \sigma } (  \psi(\tau) ,\dfrac{[\sigma : \tau] m(\sigma )}{m(\tau ) (k+2)!} \phi ( \sigma ) ) = 
 \sum_{\tau \in \Sigma (k, \Gamma)} \dfrac{m(\tau )}{(k+1)! \vert \Gamma_\tau \vert}  ( \psi(\tau) ,\sum_{\sigma \in\Sigma(k+1), \tau \subset \sigma} \dfrac{[\sigma : \tau] m(\sigma )}{m(\tau ) (k+2)!} \phi ( \sigma ) ) = 
\sum_{\tau \in \Sigma (k, \Gamma)} \dfrac{m(\tau )}{(k+1)! \vert \Gamma_\tau \vert}  (  \psi(\tau) ,\sum_{ v \in\Sigma(0), v\tau\in\Sigma(k+1) }\frac{m(v\tau)}{m(\tau)}\phi(v\tau) ).
\end{dmath*}

\end{enumerate}
\end{proof}

We end this section by proving the following criterion for vanishing of cohomology that appeared in a different form in \cite{Nowak} (we claim no originality here):
\begin{lemma}
\label{vanishing of coho lemma - Nowak}
Let $X, \Gamma, \B, \pi$ be as above and $1 \leq k \leq n-1$. If there is a constant $C<1$ such that for every $\phi \in C^k (X,\pi), \psi \in C^k (X,\overline{\pi})$,
$$\left\vert \langle d_k \phi, \overline{d_k} \psi \rangle \right\vert + \left\vert \langle \overline{d_{k-1}}^* \phi, d_{k-1}^* \psi \rangle \right\vert \geq  \vert \langle \phi , \psi \rangle \vert - C (\frac{\Vert \phi \Vert^2 +  \Vert \psi \Vert^2}{2}),$$
then $H^k (X,\pi) = H^k (X,\overline{\pi}) = 0$.  
\end{lemma}

Before proving this Lemma, we recall the following facts regarding adjoint operators (for proof of these facts, see for instance \cite[Corollary 1.6.6, Theorem 3.1.22]{IntroToBanachBook}):
\begin{theorem}
\label{T,T^* onto fact}
Let $\B_1, \B_2$ be Banach spaces and $T: \B_1 \rightarrow \B_2$ be a bounded linear operator. Then
\begin{enumerate}
\item The following are equivalent:
\begin{enumerate}
\item $T$ maps $\B_1$ onto $\B_2$.
\item $T^*$ is an isomorphism from $\B_2^*$ onto a subspace of $\B_1^*$.
\item There is a constant $c>0$ such that for every $x\in \B_2^*$, $\Vert T^* x \Vert \geq c \Vert x \Vert$.
\item $T^*$ is injective with a closed image.
\end{enumerate}
\item The following are equivalent:
\begin{enumerate}
\item $T^*$ maps $\B_2^*$ onto $\B_1^*$.
\item $T$ is an isomorphism from $\B_1$ onto a subspace of $\B_2$.
\item There is a constant $c>0$ such that for every $x\in \B_1$, $\Vert T x \Vert \geq c \Vert x \Vert$.
\item $T$ is injective with a closed image.
\end{enumerate} 
\end{enumerate}
\end{theorem}

Using these facts, we can prove the Lemma \ref{vanishing of coho lemma - Nowak}:
\begin{proof}
We will only prove that $H^k (X,\pi) = 0$ - the proof for $H^k (X,\overline{\pi})$ is similar. We denote $d_{k-1} '$ to be the $k-1$ differential with range $\Ker (d_k)$, i.e.,  $d_{k-1} ' : C^{k-1} (X, \pi) \rightarrow \Ker (d_k)$ and we also denote $i: \Ker (d_k) \hookrightarrow  C^{k} (X, \pi)$ to be the natural injection. Therefore $d_{k-1} = i \circ d_{k-1} '$. We similarly denote $\overline{d_{k-1}} ' : C^{k-1} (X, \overline{\pi}) \rightarrow \Ker (\overline{d_k})$ and $\overline{i} :  \Ker (\overline{d_k}) \hookrightarrow  C^{k} (X, \overline{\pi})$ and with this notation $\overline{d_{k-1}} = \overline{i} \circ \overline{d_{k-1}}'$.

By the assumptions of the Lemma, for every $\phi \in \Ker (d_k)$, taking $\psi = \phi^* \in J \phi$ (using Proposition \ref{J prop}), yields that
$$\left\vert \langle \overline{d_{k-1}}^* \phi, d_{k-1}^* \phi^* \rangle \right\vert \geq \vert \langle \phi , \phi^* \rangle \vert - C (\frac{\Vert \phi \Vert^2 +  \Vert \phi^* \Vert^2}{2}) = (1-C) \Vert \phi \Vert^2.$$
We note that by Proposition \ref{bound on d, calc of d^* prop},
\begin{dmath*}
\left\vert \langle \overline{d_{k-1}}^* \phi, d_{k-1}^* \phi^* \rangle \right\vert \leq 
\Vert \overline{d_{k-1}}^* \phi \Vert \Vert d_{k-1}^* \phi^* \Vert \leq 
\Vert \overline{d_{k-1}}^* \phi \Vert \sqrt{k+2} \Vert \phi \Vert.
\end{dmath*}
Thus, for every $\phi \in \Ker (d_k)$,
$$\Vert \overline{d_{k-1}}^* \phi \Vert \geq \frac{1-C}{\sqrt{k+2}} \Vert \phi \Vert.$$
This yields that $\overline{d_{k-1}}^* \circ i$ is injective with a closed image. By the notations above, $(\overline{d_{k-1}}')^* \circ \overline{i}^* \circ i$ is injective with a closed image, and therefore $\overline{i}^* \circ i : \Ker (d_k) \rightarrow (\Ker (\overline{d_k}))^*$ is injective with a closed image. Note that $\Ker (\overline{d_k})$ is a closed subspace of a reflexive space (using Proposition \ref{C^k is reflexive prop}) such and thus $\Ker (\overline{d_k})$ is reflexive and it follows that $(\Ker (\overline{d_k}))^*$ is reflexive as well. Therefore by Theorem \ref{T,T^* onto fact}, $i^* \circ \overline{i} = (\overline{i}^* \circ i)^* : (\Ker (d_k))^* \rightarrow \Ker (\overline{d_k})$ is onto. 

By a similar argument, for a given $\psi \in \Ker (\overline{d_k})$, if we take $\phi = \psi^* \in \overline{J} \psi$, then 
$$\left\vert \langle \overline{d_{k-1}}^* \psi^*, d_{k-1}^* \psi \rangle \right\vert \geq  (1-C) \Vert \psi \Vert^2,$$
which implies that 
$$\Vert d_{k-1}^* \psi \Vert \geq \frac{1-C}{\sqrt{k+2}} \Vert \psi \Vert.$$
Arguing as above, we deduce from this inequality that $(d_{k-1}')^* \circ i^* \circ \overline{i}$ is injective with a closed image. 

We showed above that $i^* \circ \overline{i}$ is onto and therefore it follows that $(d_{k-1}')^* : (\Ker (d_k))^* \rightarrow (C^{k-1} (X, \pi))^*$ is injective with a closed image. Thus applying Theorem \ref{T,T^* onto fact} yields that $d_{k-1}'$ is onto, i.e., $\im (d_{k-1}) = \Ker (d_k)$, or in other words, $H^k (X,\pi) =0$. 
\end{proof}

\begin{remark}
As in \cite{BS}, we can define the Laplacian operators as follows: $\Delta_k^+ = \overline{d_k}^* d_k, \Delta_k^- = d_{k-1} \overline{d_{k-1}}^*$. With these notations, the condition in Lemma \ref{vanishing of coho lemma - Nowak} can be reformulated as follows: there is a constant $C <1$ such that for every $\phi \in C^k (X,\pi), \psi \in C^k (X,\overline{\pi})$,
$$\left\vert \langle \Delta_k^+ \phi, \psi \rangle \right\vert + \left\vert \langle \Delta_k^- \phi, \psi \rangle \right\vert \geq \vert \langle \phi , \psi \rangle \vert - C (\frac{\Vert \phi \Vert^2 +  \Vert \psi \Vert^2}{2}).$$
\end{remark}

\section{Local criteria for vanishing of Banach cohomology}
\label{local criteria sec}

Below, we will prove local criteria for vanishing of equivariant cohomology in the spirit of ``Garland's method''. The method is an adaption of \cite{BS}, but unlike the case of Hilbert spaces, considered in \cite{BS}, in which the condition for vanishing of cohomology requires a (one-sided) spectral gap in the links, here the condition for vanishing of cohomology will require a two-sided spectral gap in the same links.

Let $X, \Gamma, \B, \pi$ as in \cref{equiv coho sec} (recall that we assume that $\pi$ is continuous and $\B$ is reflexive and thus $\overline{\pi}$ is also continuous). Given an ordered simplex $\left(v_{0},...,v_{j}\right)=\tau\in\Sigma(j)$, the link of $\tau$ is simply the link of $\lbrace v_{0},...,v_{j} \rbrace$  defined above. Below, we will only by interested in the $1$-skeleton to the links: given $\tau\in\Sigma(j)$, the \textit{$1$-skeleton of} $X_\tau$ is the weighted graph, denoted $(V_\tau, E_\tau)$, defined as
$$V_\tau = \lbrace v : \lbrace v \rbrace \in X_\tau (0) \rbrace, E_\tau = X_\tau (1),$$
with the weight function $m_\tau (\lbrace u,v \rbrace) =  m (\tau \cup \lbrace u,v \rbrace)$, where $\tau \cup \lbrace u,v \rbrace$ is defined by the abuse of notation of treating $\tau$ as a set (and forgetting the ordering), i.e., $m( (v_{0},...,v_{j}) \cup \lbrace u,v \rbrace) = m(\lbrace v_{0},...,v_{j}, u,v\rbrace)$. Note that with this definition, $m_\tau (v) = m(\tau \cup \lbrace v \rbrace)$. 

On this weighted graph, we define $\ell^2 (V_\tau, m_\tau), \ell^2 (V_\tau, m_\tau ; \B)$ and the operators $A_\tau$, $M_\tau$ as in \cref{rw on graph subsec}. On $\ell^2 (V_\tau, m_\tau ; \B)$ define a norm denoted $\Vert . \Vert_\tau$ as in \cref{equiv coho sec}, i.e., for $\phi \in \ell^2 (V_\tau, m_\tau ; \B)$,
$$\Vert \phi \Vert_\tau = \left( \sum_{v \in V_\tau} m_{\tau} (v) \vert \phi (v) \vert^2 \right)^{\frac{1}{2}},$$
where $\vert . \vert$ is the norm of $\B$. 
Also, define a coupling $\langle .,. \rangle_\tau$ between $\ell^2 (V_\tau, m_\tau ; \B)$ and $\ell^2 (V_\tau, m_\tau ; \B^*)$ as follows: for $\phi \in \ell^2 (V_\tau, m_\tau ; \B), \psi \in \ell^2 (V_\tau, m_\tau ; \B^*)$,
$$\langle \phi, \psi \rangle_\tau =  \sum_{v \in V_\tau} m_{\tau} (v) ( \phi (v), \psi (v)),$$
where $(.,.)$ is the standard coupling between $\B$ and $\B^*$.

Given $\phi \in C^k (X,\pi)$ and $\tau \in \Sigma (k-1)$ we define the \textit{localization of $\phi$ at $X_\tau$}, denoted $\phi_\tau \in \ell^2 (V_\tau, m_\tau ; \B)$, as
$$\phi_\tau (v) = \phi (v \tau), \forall v \in V_\tau,$$
where $v \tau$ is the concatenation of $v$ with $\tau$, i.e., for $\tau = (v_0,...,v_{k-1})$, $v \tau = (v, v_0,...,v_{k-1})$. We note that by the definition of $X_\tau$, $v \tau \in \Sigma (k)$ and therefore $\phi (v \tau)$ is well-defined.

The basic observation of Garland in \cite{Gar} was that the norm of cochains can be computed by considering their localizations. Below, we generalize this observation to the Banach setting. The calculations below are very similar to those of \cite{BS}, but we included all the calculations, because we need localization results not only for the norms, but for the couplings.

\begin{lemma}
\label{inner-prod and norm calc lemma}
Let $1 \leq k \leq n-1$, $\phi \in C^k (X, \pi), \psi \in C^k (X, \overline{\pi})$. Then
$$(k+1)! \langle \phi , \psi \rangle = \sum_{\tau \in \Sigma (k-1,\Gamma)} \frac{1}{\vert \Gamma_\tau \vert} \langle \phi_\tau , \psi_\tau \rangle_\tau,$$
$$(k+1)! \Vert \phi \Vert^2 = \sum_{\tau \in \Sigma (k-1,\Gamma)} \frac{1}{\vert \Gamma_\tau \vert} \Vert \phi_\tau \Vert_\tau^2,$$
and
$$(k+1)! \Vert \psi \Vert^2 = \sum_{\tau \in \Sigma (k-1,\Gamma)} \frac{1}{\vert \Gamma_\tau \vert} \Vert \psi_\tau \Vert_\tau^2.$$
\end{lemma}

\begin{proof}
All these equalities follow from the definition of the localization and Proposition \ref{changing the order of sum prop} and thus we will only prove the first equality, leaving the other two for the reader. Fix  $\phi \in C^k (X, \pi), \psi \in C^k (X, \overline{\pi})$, then 
\begin{dmath*}
\sum_{\tau \in \Sigma (k-1,\Gamma)} \frac{1}{\vert \Gamma_\tau \vert} \langle \phi_\tau , \psi_\tau \rangle_\tau = 
\sum_{\tau \in \Sigma (k-1,\Gamma)} \frac{1}{\vert \Gamma_\tau \vert} \sum_{v \in V_\tau} m_\tau (v) ( \phi_\tau (v), \psi_\tau (v) ) = \\
\sum_{\tau \in \Sigma (k-1,\Gamma)} \frac{1}{\vert \Gamma_\tau \vert} \sum_{v \in V_\tau} m (v \tau ) ( \phi (v \tau ), \psi (v \tau) ) = \\
\sum_{\tau \in \Sigma (k-1,\Gamma)} \frac{1}{\vert \Gamma_\tau \vert} \frac{1}{(k+1)!} \sum_{\sigma \in \Sigma (k), \tau \subseteq \sigma} m (\sigma) ( \phi (\sigma), \psi (\sigma) ) =  \\
\sum_{\sigma \in \Sigma (k,\Gamma)}   \frac{m(\sigma)}{(k+1)! \vert \Gamma_\sigma \vert} \sum_{\tau \in \Sigma (k-1,\Gamma)} ( \phi (\sigma), \psi (\sigma) ) \sum_{\tau \in \Sigma (k-1), \tau \subseteq \sigma} 1 = 
(k+1)! \langle \phi , \psi \rangle.
\end{dmath*}
\end{proof}

\begin{lemma}
\label{d* inner-prod calc lemma}
Let $1 \leq k \leq n-1$, $\phi \in C^k (X, \pi), \psi \in C^k (X, \overline{\pi})$. Then
$$ \langle \overline{d_{k-1}}^* \phi, d_{k-1}^* \psi \rangle =  \frac{1}{k!} \sum_{\tau \in \Sigma (k-1,\Gamma)} \frac{1}{\vert \Gamma_\tau \vert} \langle (M_\tau \otimes \id_{\B}) \phi_\tau, \psi_\tau \rangle_\tau.$$
\end{lemma}

\begin{proof}
By Proposition \ref{bound on d, calc of d^* prop}, for every $\tau \in \Sigma (k-1)$,
$$ \overline{d_k}^* \phi(\tau)= \sum_{ v \in V_\tau}\frac{m(v\tau)}{m(\tau)}\phi(v\tau), d_k^* \psi(\tau)=\sum_{ v \in V_\tau}\frac{m(v\tau)}{m(\tau)}\psi(v\tau).$$
We note that by definition $m_\tau (\emptyset) = m(\tau)$ and therefore for every $\tau \in \Sigma (k-1)$,
\begin{dmath*}
\langle (M_\tau \otimes \id_{\B}) \phi_\tau, \psi_\tau \rangle_\tau =
\sum_{v \in V_\tau} m_\tau (v) (\sum_{u \in V_\tau} \frac{m_\tau (u)}{m_\tau (\emptyset)} \phi_\tau (u), \psi_\tau (v) ) =
\sum_{v \in V_\tau} m (v \tau) (\sum_{u \in V_\tau} \frac{m (u \tau)}{m (\tau)} \phi (u \tau), \psi  (v\tau) ) =
 (\sum_{u \in V_\tau} \frac{m (u \tau)}{m (\tau)} \phi (u \tau), \sum_{v \in V_\tau} m (v \tau) \psi  (v\tau) ) =
m(\tau) (\overline{d_k}^* \phi(\tau), d_k^* \psi(\tau)).
\end{dmath*}
Therefore
\begin{dmath*}
\sum_{\tau \in \Sigma (k-1,\Gamma)} \frac{1}{\vert \Gamma_\tau \vert} \langle (M_\tau \otimes \id_{\B}) \phi_\tau, \psi_\tau \rangle_\tau = \\
\sum_{\tau \in \Sigma (k-1,\Gamma)} \frac{m(\tau )}{\vert \Gamma_\tau \vert} (\overline{d_k}^* \phi(\tau), d_k^* \psi(\tau)) =
k!  \langle \overline{d_{k-1}}^* \phi, d_{k-1}^* \psi \rangle.
\end{dmath*}
\end{proof}

\begin{lemma}
\label{d inner-prod calc lemma}
Let $1 \leq k \leq n-1$, $\phi \in C^k (X, \pi), \psi \in C^k (X, \overline{\pi})$. Then
$$\langle d_k \phi, \overline{d_{k}} \psi \rangle = \langle \phi , \psi \rangle -  \frac{1}{k!} \sum_{\tau \in \Sigma (k-1,\Gamma)} \frac{1}{\vert \Gamma_\tau \vert} \langle (A_\tau \otimes \id_{\B}) \phi_\tau, \psi_\tau \rangle_\tau.$$
\end{lemma}

\begin{proof}
For $\eta = (v_0,...,v_{k+1}) \in \Sigma (k+1)$ and $0 \leq i \neq j \leq k+1$, denote  $\eta_i = (v_0,...,\widehat{v_i},...,v_{k+1})$ and $\eta_{i,j} = (v_0,...,\widehat{v_i},...,\widehat{v_j},...,v_{k+1})$. Then
\begin{dmath*}
(d_k \phi (\eta), \overline{d_{k}} \psi (\eta)) = (\sum_{i=0}^{k+1} (-1)^i \phi (\eta_i), \sum_{j=0}^{k+1} (-1)^j \psi (\eta_j)) =
\sum_{i=0}^{k+1} (\phi (\eta_i), \psi (\eta_i)) + \sum_{0 \leq i \neq j \leq k+1} (-1)^{i+j} ( \phi (\eta_i), \psi (\eta_j)).
\end{dmath*}
We note that by the assumption that $\phi, \psi$ are alternating, changing the order of $\eta_i$ in the first sum above does not change the coupling and therefore
$$\sum_{i=0}^{k+1} (\phi (\eta_i), \psi (\eta_i)) = \frac{1}{(k+1)!} \sum_{\sigma \in \Sigma (k), \sigma \subseteq \eta} (\phi (\sigma), \psi (\sigma)).$$
We also not that for every $i \neq j$, 
$$(\phi (\eta_i), \psi (\eta_j)) = (-1)^{i+j-1} (\phi (v_j \eta_{i,j}), \psi (v_i \eta_{i,j})),$$
(this can be shown by considering the cases $i<j$ and $j<i$ - we leave the proof for the reader). Therefore
\begin{dmath*}
\sum_{0 \leq i \neq j \leq k+1} (-1)^{i+j} ( \phi (\eta_i), \psi (\eta_j)) = 
- \sum_{0 \leq i \neq j \leq k+1} ( \phi (v_j \eta_{i,j}), \psi (v_i \eta_{i,j})) = 
\frac{1}{k!} \sum_{\tau \in \Sigma (k-1), \tau \subseteq \eta} \sum_{v, v\tau \subseteq \eta} (\sum_{u, u \neq v, uv\tau \subseteq \eta} \phi (u \tau), \psi (v \tau)),
\end{dmath*}
where $uv \tau$ is the concatenation, i.e., if $\tau = (v_0,...,v_{k-1})$, \\ $uv \tau = (u,v,v_0,...,v_{k-1})$ (we recall that $uv\tau \subseteq \eta$ refers only to inclusion as sets without regarding the ordering). This yields that 
\begin{dmath}
\label{general eq}
\langle d_k \phi, \overline{d_{k}} \psi \rangle = 
\sum_{\eta \in \Sigma (k+1, \Gamma)} \frac{m(\eta)}{(k+2)! \vert \Gamma_\eta \vert} \frac{1}{(k+1)!} \sum_{\sigma \in \Sigma (k), \sigma \subseteq \eta} (\phi (\sigma), \psi (\sigma)) - \sum_{\eta \in \Sigma (k+1, \Gamma)} \frac{m(\eta)}{(k+2)! \vert \Gamma_\eta \vert} \frac{1}{k!} \sum_{\tau \in \Sigma (k-1), \tau \subseteq \eta} \sum_{v, v\tau \subseteq \eta} (\sum_{u, u \neq v, uv\tau \subseteq \eta} \phi (u \tau), \psi (v \tau)).
\end{dmath}
We will calculate each one of the expressions above separately. First, by applying Proposition \ref{changing the order of sum prop},
\begin{dmath}
\label{first}
\sum_{\eta \in \Sigma (k+1, \Gamma)} \frac{m(\eta)}{(k+2)! \vert \Gamma_\eta \vert} \frac{1}{(k+1)!} \sum_{\sigma \in \Sigma (k), \sigma \subseteq \eta} (\phi (\sigma), \psi (\sigma)) = \\
\sum_{\sigma \in \Sigma (k, \Gamma)} \frac{1}{(k+1)! \vert \Gamma_\sigma \vert} (\phi (\sigma), \psi (\sigma)) \sum_{\sigma \in \Sigma (k+1), \sigma \subseteq \eta} \frac{m(\eta)}{(k+2)!} = \\
\sum_{\sigma \in \Sigma (k, \Gamma)} \frac{m(\sigma)}{(k+1)! \vert \Gamma_\sigma \vert} (\phi (\sigma), \psi (\sigma)) =
\langle \phi, \psi \rangle.
\end{dmath}
Second, applying Proposition \ref{changing the order of sum prop} to the second expression,
\begin{dmath}
\label{second}
\sum_{\eta \in \Sigma (k+1, \Gamma)} \frac{m(\eta)}{(k+2)! \vert \Gamma_\eta \vert} \frac{1}{k!} \sum_{\tau \in \Sigma (k-1), \tau \subseteq \eta} \sum_{v, v\tau \subseteq \eta} (\sum_{u, u \neq v, uv\tau \subseteq \eta} \phi (u \tau), \psi (v \tau)) = \\
\frac{1}{k!} \sum_{\tau \in \Sigma (k-1, \Gamma)} \frac{1}{ \vert \Gamma_\tau \vert} \sum_{\eta \in \Sigma (k+1), \tau \subseteq \eta} \frac{m(\eta)}{(k+2)!} \sum_{v, v\tau \subseteq \eta} (\sum_{u, u \neq v, uv\tau \subseteq \eta} \phi (u \tau), \psi (v \tau)) = \\
\frac{1}{k!} \sum_{\tau \in \Sigma (k-1, \Gamma)} \frac{1}{ \vert \Gamma_\tau \vert} \sum_{\lbrace v,u \rbrace \in E_{\tau}} m_\tau (\lbrace v,u \rbrace) \sum_{v \in \lbrace v,u \rbrace} (\sum_{u \in \lbrace v,u \rbrace, u \neq v} \phi_\tau (u), \psi_\tau (v)) = \\
\frac{1}{k!} \sum_{\tau \in \Sigma (k-1, \Gamma)} \frac{1}{ \vert \Gamma_\tau \vert} \sum_{v \in V_\tau}  (\sum_{u \in \lbrace v,u \rbrace, u \neq v}  m_\tau (\lbrace v,u \rbrace) \phi_\tau (u), \psi_\tau (v)) = \\
\frac{1}{k!} \sum_{\tau \in \Sigma (k-1, \Gamma)} \frac{1}{ \vert \Gamma_\tau \vert} \sum_{v \in V_\tau} m_\tau (v) ( \sum_{u \in \lbrace v,u \rbrace, u \neq v}  \frac{m_\tau (\lbrace v,u \rbrace)}{m_\tau (v)} \phi_\tau (u), \psi_\tau (v)) = \\
\frac{1}{k!} \sum_{\tau \in \Sigma (k-1, \Gamma)} \frac{1}{ \vert \Gamma_\tau \vert} \langle (A_\tau \otimes \id_{\B}) \phi_\tau, \psi_\tau \rangle_{\tau}.
\end{dmath}
Combining \eqref{general eq}, \eqref{first}, \eqref{second} yields the needed equality.
\end{proof}

After these lemmata, we can prove a local criterion for cohomology vanishing that appeared as Theorem \ref{local criterion thm - intro} in the introduction:
\begin{theorem}
\label{Local criterion thm}
Let $X$ be a locally finite, pure $n$-dimensional simplicial complex with the weight function $m$ defined above and $\Gamma$ be a locally compact, unimodular group acting cocompactly and properly on $X$. For every reflexive Banach space $\B$ and every $1 \leq k \leq n-1$, if  
$$\max_{\tau \in \Sigma (k-1, \Gamma)} \Vert (A_\tau (I-M_\tau) \otimes \id_{\B}) \Vert_{B (\ell^2 (V_\tau,m_\tau ; \B))} < \frac{1}{k+1},$$
then for every continuous isometric representation $\pi$ of $\Gamma$ on $\B$ it holds that $H^k (X, \pi) = 0$. 
\end{theorem}

\begin{proof}
Let $\B$ be a reflexive Banach space and $\pi$ be a continuous isometric representation of $\Gamma$ on $\B$. Denote 
$$C' = \max_{\tau \in \Sigma (k-1, \Gamma)} \Vert (A_\tau (I-M_\tau) \otimes \id_{\B}) \Vert_{B (\ell^2 (V,m ; \B))}.$$
The by Lemma \ref{d inner-prod calc lemma}, for every $\phi \in C^k (X, \pi), \psi \in C^k (X, \overline{\pi})$, 
\begin{dmath*}
\langle d_k \phi, \overline{d_{k}} \psi \rangle = \langle \phi , \psi \rangle -  \frac{1}{k!} \sum_{\tau \in \Sigma (k-1,\Gamma)} \frac{1}{\vert \Gamma_\tau \vert} \langle (A_\tau \otimes \id_{\B}) \phi_\tau, \psi_\tau \rangle_\tau = 
\langle \phi , \psi \rangle -  \frac{1}{k!} \sum_{\tau \in \Sigma (k-1,\Gamma)} \frac{1}{\vert \Gamma_\tau \vert} \langle (A_\tau (I-M_\tau) \otimes \id_{\B})  \phi_\tau, \psi_\tau \rangle_\tau \\
- \frac{1}{k!} \sum_{\tau \in \Sigma (k-1,\Gamma)} \frac{1}{\vert \Gamma_\tau \vert} \langle (A_\tau M_\tau \otimes \id_{\B}) \phi_\tau, \psi_\tau \rangle_\tau =^{A_\tau M_{\tau} = M_{\tau}} 
\langle \phi , \psi \rangle -  \frac{1}{k!} \sum_{\tau \in \Sigma (k-1,\Gamma)} \frac{1}{\vert \Gamma_\tau \vert} \langle (A_\tau (I-M_\tau) \otimes \id_{\B})  \phi_\tau, \psi_\tau \rangle_\tau \\
- \frac{1}{k!} \sum_{\tau \in \Sigma (k-1,\Gamma)} \frac{1}{\vert \Gamma_\tau \vert} \langle (M_\tau \otimes \id_{\B}) \phi_\tau, \psi_\tau \rangle_\tau  =^{\text{Lemma } \ref{d* inner-prod calc lemma}} 
\langle \phi , \psi \rangle -  \frac{1}{k!} \sum_{\tau \in \Sigma (k-1,\Gamma)} \frac{1}{\vert \Gamma_\tau \vert} \langle (A_\tau (I-M_\tau) \otimes \id_{\B})  \phi_\tau, \psi_\tau \rangle_\tau \\
- \langle \overline{d_{k-1}}^* \phi, d_{k-1}^* \psi \rangle.
\end{dmath*}
Thus
\begin{dmath*}
\langle d_k \phi, \overline{d_{k}} \psi \rangle + \langle \overline{d_{k-1}}^* \phi, d_{k-1}^* \psi \rangle = \langle \phi , \psi \rangle -  \frac{1}{k!} \sum_{\tau \in \Sigma (k-1,\Gamma)} \frac{1}{\vert \Gamma_\tau \vert} \langle (A_\tau (I-M_\tau) \otimes \id_{\B})  \phi_\tau, \psi_\tau \rangle_\tau.
\end{dmath*}
Applying absolute value on this equation and using the triangle inequality,
\begin{dmath*}
{\left\vert \langle d_k \phi, \overline{d_k} \psi \rangle \right\vert + \left\vert \langle \overline{d_{k-1}}^* \phi, d_{k-1}^* \psi \rangle \right\vert \geq}\\
 \vert \langle \phi , \psi \rangle \vert -  \frac{1}{k!} \sum_{\tau \in \Sigma (k-1,\Gamma)} \frac{1}{\vert \Gamma_\tau \vert} \vert \langle (A_\tau (I-M_\tau) \otimes \id_{\B})  \phi_\tau, \psi_\tau \rangle_\tau \vert \geq \\
 \vert \langle \phi , \psi \rangle \vert -  \frac{1}{k!} \sum_{\tau \in \Sigma (k-1,\Gamma)} \frac{1}{\vert \Gamma_\tau \vert} \Vert (A_\tau (I-M_\tau) \otimes \id_{\B}) \Vert_{B (\ell^2 (V_\tau,m_\tau ; \B))} \Vert \phi_\tau \Vert_\tau  \Vert \psi_\tau \Vert_\tau \geq \\
 \vert \langle \phi , \psi \rangle \vert -  \frac{1}{k!} \sum_{\tau \in \Sigma (k-1,\Gamma)} \frac{1}{\vert \Gamma_\tau \vert} C' \frac{\Vert \phi_\tau \Vert_\tau^2+ \Vert \psi_\tau \Vert_\tau^2}{2} =^{\text{Lemma } \ref{inner-prod and norm calc lemma}} \\
  \vert \langle \phi , \psi \rangle \vert - (k+1)C'  \frac{\Vert \phi \Vert^2+ \Vert \psi \Vert^2}{2}
\end{dmath*}

If we denote $C = (k+1)C'$, then by our assumption $C<1$ and we prove that 
$$\left\vert \langle d_k \phi, \overline{d_k} \psi \rangle \right\vert + \left\vert \langle \overline{d_{k-1}}^* \phi, d_{k-1}^* \psi \rangle \right\vert \geq  \vert \langle \phi , \psi \rangle \vert - C (\frac{\Vert \phi \Vert^2 +  \Vert \psi \Vert^2}{2}),$$
and by Lemma \ref{vanishing of coho lemma - Nowak}, $H^k (X, \pi) = H^k (X, \overline{\pi}) = 0$. 
\end{proof}

Next, we will apply this Theorem in the context of uniformly curved spaces: 

\begin{proposition}
\label{vanishing of k-coho prop for uni curv}
Let $X,\Gamma$ be as above and $\alpha : (0,1] \rightarrow (0,1]$ be a strictly monotone increasing function. Fix $1 \leq k \leq n-1$. If there is $\lambda < \alpha^{-1} (\frac{1}{2(k+1)})$ such that for every $\tau \in \Sigma (k-1, \Gamma)$, the one-skeleton of $X_\tau$ is a two-sided $\lambda$-spectral expander, then for every $\B \in \mathcal{E}^{\text{u-curved}}_\alpha$ and every continuous isometric representation $\pi$ of $\Gamma$ on $\B$, $H^k(X,\pi) = 0$.
\end{proposition}

\begin{proof}
First, recall that by Theorem \ref{uni. cur. are s-reflex. thm}, every $\B \in \mathcal{E}^{\text{u-curved}}_\alpha$ is reflexive. Second, by Corollary \ref{norm bound on rw on uc coro} for every $\tau \in \Sigma (k-1, \Gamma)$,
$$\Vert (A_\tau (I-M_\tau)) \otimes \id_{\B} \Vert_{B(\ell^2 (V,m ; \B))} \leq 2 \alpha (\lambda)  < 2 \alpha \left( \alpha^{-1} \left(\frac{1}{2(k+1)} \right) \right) = \frac{1}{k+1}.$$
Therefore, the conditions of Theorem \ref{Local criterion thm} are fulfilled and for every continuous isometric representation $\pi$ of $\Gamma$ on $\B$, $H^k(X,\pi) = 0$.

\end{proof}

As a result of this Proposition we deduce the following vanishing result for strictly Hilbertian spaces that appeared in Corollary \ref{theta Hil coro - intro} (1): 

\begin{corollary}
\label{theta hil vani coro}
Let $X,\Gamma$ be as above, $0 < \theta_0 \leq 1$ a constant. Denote $\mathcal{E}_{\theta_0}$ to be the smallest class of Banach spaces that contains all strictly $\theta$-Hilbertian Banach spaces for all $\theta_0 \leq \theta \leq 1$ and is closed under passing to quotients, subspaces, $\ell^2$-sums and ultraproducts of Banach spaces. Fix $1 \leq k \leq n-1$. If there is $0 < \lambda < \left(\frac{1}{2(k+1)} \right)^{\frac{1}{\theta_0}}$ such that for every $\tau \in \Sigma (k-1, \Gamma)$, the one-skeleton of $X_\tau$ is a two-sided $\lambda$-spectral expander, then for every $\B \in \mathcal{E}_{\theta_0}$ and every continuous isometric representation $\pi$ of $\Gamma$ on $\B$, $H^k(X,\pi) = 0$.
\end{corollary}

\begin{proof}
Corollary \ref{theta-hil is in uni-curv cor} states that $\mathcal{E}_{\theta_0} \subseteq \mathcal{E}^{\text{u-curved}}_{\alpha (t) = t^{\theta_0}}$. Thus the assertion follows directly from Proposition \ref{vanishing of k-coho prop for uni curv}.
\end{proof}

Specializing this Corollary to the case of vanishing of the $L^p$ cohomology of a group acting on a $2$-dimensional simplicial complex yields:
\begin{corollary}
\label{criterion coho vanish for L^p}
Let $X$ be a locally finite, pure $2$-dimensional simplicial complex such that all the links of $X$ of dimension $\geq 1$ are connected and $\Gamma$ be a locally compact, unimodular group acting cocompactly and properly on $X$. Also let $p >2$, $0 < \lambda <  \frac{1}{2^p}$ be constants. Assume that for every vertex $\lbrace v \rbrace \in X(0)$, the one-skeleton of $X_{\lbrace v \rbrace}$ is a two-sided $\lambda$-spectral expander. Then for every $2 \leq p' \leq p$, every space $\B$ that is a commutative or non-commutative $L^{p'}$-space and every continuous isometric representation $\pi$ of $\Gamma$ on $\B$ it holds that $H^1 (X, \pi) =0$.
\end{corollary}

\begin{proof}
As noted above, for $2 \leq p < \infty$, every (commutative or non-commutative) $L^p$-space is $\theta$-Hilbertian with $\theta = \frac{2}{p}$. Thus applying Corollary \ref{theta hil vani coro} with $k=1, n=2$ and $\theta = \frac{2}{p}$ gives the stated result.
\end{proof}

The conditions for Proposition \ref{vanishing of k-coho prop for uni curv} and Corollary \ref{theta hil vani coro} can be deduced for all $1 \leq k \leq n-1$, based only on the $1$-dimensional links of $X$. This is done via the following Theorem from \cite[Theorem 1.4]{OppLocSpec}:
\begin{theorem}
\label{spec descent thm}
Let $Y$ be a finite, pure $l$-dimensional complex where $l \geq 2$, such that (one-skeletons of) all the  links of $Y$ of dimension $\geq 1$ are connected (including the one-skeleton of $Y$). Denote $m_Y$ to be the weight function on $Y$, $V_Y$ the vertices of the $1$-skeleton of $Y$ and $A_Y,M_Y$ the operators associated with the random walk on this $1$-skeleton. Let $-1 \leq \kappa_1 \leq 0 \leq  \kappa_2 \leq \frac{1}{l}$ be constants such that for every $\tau \in Y(l-2)$ the spectrum of $A_\tau$ is contained in $[\kappa_1, \kappa_2] \cup \lbrace 1 \rbrace$. Then the spectrum of the random walk on the one-skeleton of $Y$ is contained in $[\frac{\kappa_1}{1-(l-1) \kappa_1}, \frac{\kappa_2}{1-(l-1) \kappa_2}] \cup \lbrace 1 \rbrace$. Equivalently, if there are $-1 \leq \lambda_1 \leq 0 \leq \lambda_2 \leq 1$ such that for every $\tau \in Y(l-2)$ the spectrum of $A_\tau$ is contained in $[\frac{\lambda_1}{1 + (l-1) \lambda_1}, \frac{\lambda_2}{1 + (l-1) \lambda_2}] \cup \lbrace 1 \rbrace$. Then the spectrum of the random walk on the one-skeleton of $Y$ is contained in $[\lambda_1, \lambda_2] \cup \lbrace 1 \rbrace$.
\end{theorem}

\begin{remark}
In \cite{OppLocSpec} this Theorem is written in the language of spectral gaps of Laplacians, but as noted above the translation to the language of random walks is straight-forward.
\end{remark}

\begin{observation}
\label{bound on lower spec observation}
Theorem \ref{spec descent thm} is not symmetric as it may appear at first glance: while the upper bound on the spectrum of $A_Y$ deteriorates as $l$ increase, the lower bound actually improves as $l$ increases. In particular, it is always the case that the smallest eigenvalue of the one-skeleton of every graph is $\geq -1$. Thus, in the above theorem we can always take $\kappa_1 = -1$ and get that the spectrum of the random walk on the one-skeleton of $Y$ is contained in $[-\frac{1}{l}, 1]$.
\end{observation}

Using Theorem \ref{spec descent thm}, we deduce a criterion for the vanishing of all the cohomologies:
\begin{theorem}
\label{criterion coho vanish for uc}
Let $X$ be a locally finite, pure $n$-dimensional simplicial complex such that all the links of $X$ of dimension $\geq 1$ are connected and $\Gamma$ be a locally compact, unimodular group acting cocompactly and properly on $X$. Also let $\alpha : (0,1] \rightarrow (0,1]$ be a strictly monotone increasing function, $1 \leq k \leq n-1$ and $0 < \lambda < \alpha^{-1} (\frac{1}{2(k+1)})$ be constants.
\begin{enumerate}
\item If for every $\tau \in \Sigma (n-2, \Gamma)$, the one-skeleton of $X_\tau$ is a two-sided $\frac{\lambda}{1 + (n-k-1) \lambda}$-spectral expander, then for every $\B \in \mathcal{E}^{\text{u-curved}}_\alpha$ and every continuous isometric representation $\pi$ of $\Gamma$ on $\B$, $H^k(X,\pi) = 0$.
\item If $k \leq n - \frac{1}{\lambda}$ and for every $\tau \in \Sigma (n-2, \Gamma)$, the one-skeleton of $X_\tau$ is a \emph{one-sided} $\frac{\lambda}{1 + (n-k-1) \lambda}$-spectral expander, then for every $\B \in \mathcal{E}^{\text{u-curved}}_\alpha$ and every continuous isometric representation $\pi$ of $\Gamma$ on $\B$, $H^k(X,\pi) = 0$.
\end{enumerate}
\end{theorem}

\begin{proof}
Let $1 \leq k \leq n-1$ and let $\eta \in \Sigma (k-1,\Gamma)$. If we denote $Y= X_\eta$, then $Y$ is a pure $(n-k)$-dimensional finite simplicial complex and with the notation of Theorem \ref{spec descent thm},
$$\Vert A_\eta (I-M_\eta) \Vert_{B(\ell^2 (V_\eta, m_\eta))} = \Vert A_Y (I-M_Y) \Vert_{B(\ell^2 (V_Y, m_Y))}.$$
Note that the $1$-dimensional links of $Y$ are also $1$-dimensional links of $X$. We also note that for every $\tau \in \Sigma (n-2, \Gamma)$, $X_\tau$ is a graph and $A_\tau$ is the simple random walk on this graph. 

\textbf{Case (1):} Assume that there is $0 \leq \lambda < \alpha^{-1} (\frac{1}{2(k+1)})$  such that for every $\tau \in \Sigma (n-1, \Gamma)$, the one-skeleton of $X_\tau$ is a two-sided $\frac{\lambda}{1 + (n-k-1) \lambda}$-spectral expander. Applying Theorem \ref{spec descent thm} yields that for every $\eta \in \Sigma (k-1,\Gamma)$, the one-skeleton of $X_\eta$ is a two-sided $\lambda$-spectral expander and thus the conditions of Proposition \ref{vanishing of k-coho prop for uni curv} are fulfilled and therefore for every $\B \in \B \in \mathcal{E}^{\text{u-curved}}_\alpha$ and every continuous isometric representation $\pi$ of $\Gamma$ on $\B$, $H^k(X,\pi) = 0$. 

\textbf{Case (2):} The proof is similar to Case (1), but we use Observation \ref{bound on lower spec observation} in order to bound the spectrum from below. We leave the details to the reader.
\end{proof}

Applying the above Theorem for strictly $\theta_0$-Hilbertian (with $\alpha (t) = t^{\theta_0}$) immediately yields the follow Corollary that appeared in the introduction as part of Corollary \ref{theta Hil coro - intro}: 

\begin{corollary}
\label{criterion coho vanish for Hilbertian}
Let $X$ be a locally finite, pure $n$-dimensional simplicial complex such that all the links of $X$ of dimension $\geq 1$ are connected and $\Gamma$ be a locally compact, unimodular group acting cocompactly and properly on $X$. Also let $0 < \theta_0 \leq 1$, $1 \leq k \leq n-1$, $0 < \lambda <  (\frac{1}{2(k+1)})^{\frac{1}{\theta_0}}$ be constants. Denote $\mathcal{E}_{\theta_0}$ to be the smallest class of Banach spaces that contains all strictly $\theta$-Hilbertian Banach spaces for all $\theta_0 \leq \theta \leq 1$ and is closed under subspaces, quotients, $\ell^2$-sums and ultraproducts of Banach spaces.
\begin{enumerate}
\item If for every $\tau \in \Sigma (n-2, \Gamma)$, the one-skeleton of $X_\tau$ is a two-sided $\frac{\lambda}{1 + (n-k-1) \lambda}$-spectral expander, then for every $\B \in \mathcal{E}_{\theta_0}$ and every continuous isometric representation $\pi$ of $\Gamma$ on $\B$, $H^k(X,\pi) = 0$.
\item If $k \leq n - \frac{1}{\lambda}$ and for every $\tau \in \Sigma (n-2, \Gamma)$, the one-skeleton of $X_\tau$ is a \emph{one-sided} $\frac{\lambda}{1 + (n-k-1) \lambda}$-spectral expander, then for every $\B \in \mathcal{E}_{\theta_0}$ and every continuous isometric representation $\pi$ of $\Gamma$ on $\B$, $H^k(X,\pi) = 0$.
\end{enumerate}
\end{corollary}

\bibliographystyle{alpha}
\bibliography{bibl}
\Addresses
\end{document}